\documentclass[a4paper, 11pt]{article}

\usepackage[small]{titlesec}
\usepackage{tikz-cd}
\usepackage{booktabs}
\usepackage[T1]{fontenc}
\usepackage[utf8]{inputenc}
\usepackage[english]{babel}
\usepackage[a4paper,top=4cm,bottom=3cm,left=2.5cm,right=2.5cm]{geometry}
\usepackage{amsfonts}
\usepackage{mathtools}
\usepackage{amsmath}
\usepackage{xcolor}

\usepackage{lmodern}
\usepackage{empheq}
\usepackage{amsthm}
\usepackage{amssymb}
\usepackage{booktabs}
\usepackage{caption}
\usepackage{siunitx}
\usepackage{float}
\usepackage{dsfont}
\usepackage{amssymb}
\usepackage{graphicx}
\usepackage{stmaryrd}

\usepackage{color}
\usepackage{pifont}%
\usepackage{enumerate}
\usepackage{enumitem}
\setlist[description]{leftmargin=\parindent,labelindent=\parindent}
\usepackage{listings} 
\usepackage{amsmath}
\usepackage{amscd}
\usepackage{braket}
\usepackage{tikz}
\usepackage{calc}
\usepackage[hidelinks]{hyperref}
\usepackage{resmes}
\usepackage{mathtools}

\newcommand{\R}{\mathbb{R}}
\newcommand{\sph}{\mathbb{S}}
\newcommand{\K}{\mathcal{K}}

\newcommand{\haus}{\mathcal{H}}

\newcommand{\tra}{\mathrm{tr}}
\newcommand{\Id}{\mathrm{Id}}
\newcommand{\rhobar}{\overline{\rho}}

\newtheorem{theorem}{Theorem}[section]
\newtheorem*{theorem*}{Theorem}
\newtheorem{corollary}[theorem]{Corollary}
\newtheorem{proposition}[theorem]{Proposition}
\newtheorem{lemma}[theorem]{Lemma}

\theoremstyle{definition}

\newcommand*\sq{\mathbin{\vcenter{\hbox{\rule{.3ex}{.3ex}}}}}

\usepackage{pgfplots}
\pgfplotsset{compat=1.17}

\newcommand{\Conv}{\mathrm{Conv}}
\newcommand{\vol}{\mathrm{vol}}

\newcommand{\epi}{\mathrm{epi}}

\newcommand{\di}{\, \mathrm{d}}

\newcommand{\F}{\mathcal{S}}

\usepackage{tikz-cd} 

\newcommand\blfootnote[1]{%
	\begingroup
	\renewcommand\thefootnote{}\footnote{#1}%
	\addtocounter{footnote}{-1}%
	\endgroup
}
\title{A new infinitesimal form of the Pr\'ekopa--Leindler inequality with multiplicative structure and applications}
\author{Sotiris Armeniakos and Jacopo Ulivelli}
\date{}

\begin{document}
	\maketitle
	
	\begin{abstract}
		By differentiating a concavity principle arising from the Prékopa--Leindler inequality, we obtain a statement simultaneously strengthening the weighted boundary Poincar\'e inequality and the Brascamp--Lieb variance inequality. The resulting inequality possesses a multiplicative structure, which we exploit to develop an alternative to the (by now classical) $L_2$ method in the study of geometric and analytic inequalities. We apply this approach to derive a stability estimate for the weighted Poincar\'e inequality and to investigate the dimensional Brunn--Minkowski conjecture. In particular, in the latter setting, we obtain new reformulations together with several partial results.
	\end{abstract}

\section{Introduction}

Our investigation begins with the celebrated Pr\'ekopa--Leindler inequality \cite{Leindler, Prekopa}. Consider three integrable functions $f,g,h: \R^n \to [0,\infty)$ such that, for every $t \in [0,1]$, 
\begin{equation*}
    h((1-t)x+ty) \geq f(x)^{(1-t)}g(y)^t \quad \text{ for every } x,y \in \R^n.
\end{equation*}
Then, 
\begin{equation}\label{eq:PL_1}
    \int_{\R^n} h \di x \geq \left( \int_{\R^n} f \di x \right)^{(1-t)} \left( \int_{\R^n} g \di x \right)^t \quad \text{ for every } t \in [0,1],
\end{equation}
where integration in $\di x$ is understood with respect to the Lebesgue measure. This inequality is deeply connected to the Brunn--Minkowski inequality, of which it can be considered the functional analogue. See Gardner's survey \cite{Gardner_BM} for more details on these two inequalities, and Schneider's monograph \cite{SchneiderConvexBodiesBrunn2013} for a thorough introduction to the Brunn--Minkowski theory. 

The first immediate corollary of \eqref{eq:PL_1} is that, thanks to a well-known classification result by Borell \cite{Borell}, a measure $\mu$ on $\R^n$ such that $\di \mu(x)=e^{-u(x)}\di x$, where $u:\R^n \to \R \cup \{\infty\}$ is a lower semi-continuous convex function, is log-concave. That is, \begin{equation}\label{eq:log_concave_measure} \mu((1-t)K+tL) \geq \mu(K)^{(1-t)}\mu(L)^t \end{equation} for every non-empty and compact convex sets $K,L \in \R^n$ and $t \in [0,1]$. The set-operation $+$ denotes the usual Minkowski addition (see, for example, \cite[Section 3]{SchneiderConvexBodiesBrunn2013}). If $\mu$ is chosen to be the Lebesgue measure on $\R^n$, \eqref{eq:log_concave_measure} is an equivalent form of the Brunn--Minkowski inequality.

A deeper consequence of the log-concavity property \eqref{eq:log_concave_measure} is the validity of certain Poincar\'e-type inequalities on the boundaries of convex sets. As before, suppose that $\di \mu(x)=e^{-u(x)}\di x$ where $u:\R^n \to \R \cup \{\infty\}$ is a lower semi-continuous convex function. Assume in addition that $u\in C^1(\R^n)$ and consider a compact convex set $K \subset \R^n$ whose boundary $\partial K$ is a strictly convex $C^2$ hypersurface. We shall say, with a slight abuse of terminology, that $\partial K$ is strictly convex whenever its second fundamental form ${\rm II}$ is positive definite. Then, for  every locally Lipschitz function $\rho:\partial K\to\R$, the following inequality holds:
\begin{equation}\label{eq:Poincare}
    \int_{\partial K} {\rm H}_\mu \rho^2 \di \mu-\frac{1}{\mu(K)}\left(\int_{\partial K} \rho \di \mu\right)^2\leq\int_{\partial K} \langle {\rm II}^{-1} \nabla_{\partial K} \rho, \nabla_{\partial K} \rho \rangle \di \mu .
\end{equation}
Here, $\nabla_{\partial K}$ denotes the tangential gradient, and $\langle\cdot,\cdot\rangle$ the restriction of the Euclidean scalar product to the tangent bundle of $\partial K$. With a slight abuse of notation, we use the same symbol for the scalar product on the whole space. The weighted mean curvature is given by ${\rm H}_\mu=\tra({\rm II})-\langle \nabla u,\nu_K\rangle$, where $\nu_K$ denotes the outer unit normal to $\partial K$. Integration on $\partial K$ with respect to $\mu$ is understood as integration against the $(n-1)$-dimensional Hausdorff measure weighted by $e^{-u}$. An equivalent form of \eqref{eq:Poincare} was proved by Colesanti \cite{Colesanti_poincare} in the case of the Lebesgue measure, and extended to a much more general setting by Kolesnikov and Milman \cite{Kol_Mil_P}.

Another consequence of \eqref{eq:PL_1} we are interested in is the eponymous variance inequality proved by Brascamp and Lieb in \cite{Brascamp_Lieb_76}, which we later refer to as Brascamp--Lieb inequality, for short. Consider $u \in C^2(\R^n)$ such that the Hessian matrix of $u$, which we denote by $\nabla^2 u$, is positive definite. Then, for every locally Lipschitz function $\varphi: \R^n \to \R$ and compact convex set $K \subset \R^n$ with non-empty interior, the inequality 
\begin{equation}\label{eq:Brascamp_Lieb}
    \int_K \varphi^2 \di \mu - \frac{1}{\mu(K)}\left( \int_K \varphi \di \mu\right)^2 \leq \int_K \langle (\nabla^2 u)^{-1} \nabla \varphi, \nabla \varphi\rangle \di \mu
\end{equation}
holds. A far-reaching generalization of \eqref{eq:Brascamp_Lieb} was obtained by Kolesnikov and Milman in \cite{Kol_Mil_BL}. 

\medskip

Both \eqref{eq:Poincare} and \eqref{eq:Brascamp_Lieb} can be obtained from \eqref{eq:PL_1} as infinitesimal forms. This process can be summarized as follows: Consider a function $\Phi: \R^n \times [0,1] \to \R \cup \{ \infty\}$ such that it is jointly convex. Then, \eqref{eq:PL_1} implies that  \[ \F(t)=\log \left(\int_{\R^n} e^{-\Phi(t,x)} \di x \right) \quad \text{ is concave for } t \in [0,1]. \] Statements of this kind are known as concavity principles for marginals. For more details on the topic, see, for example, the recent work of Cordero-Erausquin and Eskenazis \cite{cordero2025concavity}, which was of great inspiration in the preparation of this manuscript. If one proves that $\F$ is twice differentiable as $t \to 0^+$, the concavity of $\F$ can be equivalently stated as $\F''(0)\leq 0$. Writing explicitly $\F''(0)$ gives, for appropriate choices of $\Phi$, the inequalities \eqref{eq:Poincare} and \eqref{eq:Brascamp_Lieb}.

\medskip

In this work, after deducing a new infinitesimal form for \eqref{eq:PL_1}, we will apply this machinery to obtain a stability estimate for \eqref{eq:Poincare}. In turn, with the tools thus developed, we will show a series of applications in the direction of the dimensional Brunn--Minkowski conjecture (see later \eqref{eq:dim_BM}).

\paragraph{A new infinitesimal form of the Pr\'ekopa--Leindler inequality.} We have summarized above how the Pr\'ekopa--Leindler inequality can be differentiated to obtain two kinds of Poincar\'e-type inequalities. The first contribution of this paper is showing that \eqref{eq:Poincare} and \eqref{eq:Brascamp_Lieb} appear simultaneously (together with some new terms!) by differentiation of an appropriate choice of $\Phi$ in $\F$. Before stating this result, we introduce some new notation. Consider two Lipschitz functions $\rho_0, \rho_1 : \partial K \to \R$. Then, for $K$ and $u$ fixed and satisfying the previous hypotheses, \eqref{eq:Poincare} implies that the symmetric bilinear form 
\begin{equation}\label{eq:P}
     \langle \rho_0,\rho_1 \rangle_{\rm P}= \int_{\partial K} \langle {\rm II}^{-1} \nabla_{\partial K} \rho_0, \nabla_{\partial K} \rho_1 \rangle \di \mu - \int_{\partial K} {\rm H}_\mu \rho_0 \rho_1 \di \mu + \frac{1}{\mu(K)}\left(\int_{\partial K} \rho_0 \di \mu\right)\left(\int_{\partial K} \rho_1 \di \mu\right)
     \tag{P}
\end{equation}
is positive semi-definite. Similarly, for two locally Lipschitz functions $\varphi_0, \varphi_1: \R^n \to \R$, \eqref{eq:Brascamp_Lieb} implies that the symmetric bilinear form 
\begin{equation}\label{eq:BL}
    \langle \varphi_0, \varphi_1 \rangle_{\rm BL}= \int_K \langle (\nabla^2 u)^{-1} \nabla \varphi_0, \nabla \varphi_1\rangle \di \mu - \int_K \varphi_0 \varphi_1 \di \mu + \frac{1}{\mu(K)}\left(\int_K \varphi_0 \di \mu \right)\left( \int_K \varphi_1 \di \mu \right)
    \tag{BL}
\end{equation}
is positive semi-definite. 

Before presenting our first main result, we need a further bilinear form. Consider a Lipschitz function $\rho: \partial K \to \R$ and a locally Lipschitz function $\varphi: \R^n \to \R$. Under the regularity assumptions on $u$ and $K$ needed for \eqref{eq:Poincare} and \eqref{eq:Brascamp_Lieb}, we define 
\begin{equation}\label{eq:I}
    \langle \rho, \varphi \rangle_{\rm I} = \int_{\partial K} \rho \varphi \di \mu - \frac{1}{\mu(K)}\left( \int_{\partial K} \rho \di \mu \right) \left( \int_K \varphi \di \mu \right),
    \tag{I}
\end{equation}
which can be understood as an interaction term between \eqref{eq:Poincare} and \eqref{eq:Brascamp_Lieb} (justifying the choice of the subscript). Our first result, which takes the form of a Cauchy--Schwarz-type inequality, reads as follows.
\begin{theorem}\label{thm:multiplicative_form}
    Consider a non-empty compact convex set $K \subset \R^n$ such that $\partial K$ is a strictly convex manifold of class $C^2$, and a convex function $u \in C^2(\R^n)$ such that its Hessian matrix is positive definite. Then for every Lipschitz function $\rho: \partial K \to \R$ and every locally Lipschitz function $\varphi: \R^n \to \R$ the inequality 
    \begin{equation}\label{eq:multiplicative_form}
        \langle \rho, \varphi \rangle_{\rm I}^2 \leq \langle \rho, \rho \rangle_{\rm P} \langle \varphi, \varphi \rangle_{\rm BL}
    \end{equation}
    holds.
\end{theorem}
We will infer \eqref{eq:multiplicative_form} from another statement (see Theorem \ref{thm:mean_form}) which, under the same assumptions of Theorem \ref{thm:multiplicative_form}, reads as \[ \langle \rho, \varphi \rangle_{\rm I} \leq \frac{\langle \rho, \rho \rangle_{\rm P} + \langle \varphi, \varphi \rangle_{\rm BL}}{2}.\] From this statement it is immediate to derive \eqref{eq:Poincare} and \eqref{eq:Brascamp_Lieb} choosing $\varphi \equiv 0$ and $\rho \equiv 0$, respectively. The inequality \eqref{eq:multiplicative_form}, instead, does not look like a direct generalization of \eqref{eq:Poincare} and \eqref{eq:Brascamp_Lieb}. However, it yields a sharper lower bound. Indeed, once we know that $\langle \cdot ,\cdot \rangle_{\rm P} \geq 0 $ we obtain, for example, $\langle\varphi,\varphi\rangle_{\rm BL} \geq \langle \rho,\varphi\rangle_{\rm I}^2/\langle \rho,\rho\rangle_{\rm P}$ for every choice of $\rho$ such that $\langle \rho,\rho\rangle _{\rm P} \neq 0$, which improves the non-negativity implied by \eqref{eq:Brascamp_Lieb}. We shall see in the sequel that any $\rho\not\equiv 0$ does satisfy $\langle \rho,\rho \rangle_{\rm P} >0$. That is, $\langle \cdot , \cdot \rangle_{\rm P}$ is positive definite. On the other hand, Livshyts \cite{GalynaStability} showed that when the convex set $K$ in \eqref{eq:Brascamp_Lieb} is compact, $\langle \varphi, \varphi \rangle_{\rm BL} =0$ if and only if $\varphi$ is constant, and we may infer a similar improvement of the Poincar\'e inequality testing \eqref{eq:multiplicative_form} against any non-constant function $\varphi$. We do not investigate the equality cases of \eqref{eq:multiplicative_form}. Nevertheless, in Proposition~\ref{prop:equality}  we show that this new inequality admits additional equality cases beyond those implied by \eqref{eq:Poincare} and \eqref{eq:Brascamp_Lieb}.

The proof of Theorem \ref{thm:multiplicative_form} is based on a perturbative argument developed by the second-named author in \cite{Ulivelli_Wulff}, exploiting the connection between Wulff shapes of convex bodies and the Fenchel--Legendre transform of convex functions. The perturbations thus obtained allow the simultaneous treatment of \eqref{eq:Poincare} and \eqref{eq:Brascamp_Lieb}, highlighting the similarities between the two inequalities. This method has already found applications to geometric inequalities in \cite{MussnigUlivelli_inequalities}, by Mussnig and the second-named author.

\paragraph{Stability estimates and the dimensional Brunn--Minkowski conjecture.}After obtaining Theorem \ref{thm:multiplicative_form}, in the remainder of this manuscript, we focus on a series of applications. First, we shall prove a stability estimate for \eqref{eq:Poincare}. In turn, the results thus obtained will provide the background for some surprising applications to the dimensional Brunn--Minkowski conjecture. Before proceeding further with the exposition, we report some brief historical context for the topic. Further details can be found in the references in the following paragraph.

\medskip 

The Poincar\'e inequality \eqref{eq:Poincare} and the Brascamp--Lieb inequality \eqref{eq:Brascamp_Lieb}, together with suitable variants, have been the subject of intensive research in the last couple of decades. As mentioned above, Colesanti \cite{Colesanti_poincare} obtained a non-weighted version of \eqref{eq:Poincare}. In the same period, Bobkov and Ledoux \cite{Bobkvo_Ledoux, Bobkov_Ledoux_weight} generalized \eqref{eq:Brascamp_Lieb} to $\beta$-concave functions computing an infinitesimal form of the Borell--Brascamp--Lieb inequality (compare \cite{Borell, Brascamp_Lieb_76}), a generalization of the Pr\'ekopa--Leindler inequality \eqref{eq:PL_1}. We point out that our methods can be easily extended to obtain a generalization of Theorem \ref{thm:multiplicative_form} stemming from the Borell--Brascamp--Lieb inequality. Nonetheless, we decided to focus on the log-concave case for a cleaner and more accessible exposition. Around the same time, Cordero-Erausquin, Fradelizi, and Maurey \cite{Cordero-BGaussian} used \eqref{eq:Brascamp_Lieb} to prove the (B)-conjecture for the Gaussian measure, pioneering the so-called $L_2$ method in the context of geometric and analytic inequalities. Almost a decade later, Kolesnikov and Milman \cite{Kol_Mil_BL, Kol_Mil_P} extended this approach to the setting of weighted Riemannian manifolds under suitable curvature assumptions, reigniting interest in the topic. With similar techniques, Nguyen \cite{Nguyen_BL} extended the results of \cite{Bobkvo_Ledoux}. Afterwards, the application of the $L_2$ method has found several applications in the extension of the (B)-conjecture to even log-concave probability measures \cite{cordero2025concavity, B-conjecture, rot_B} (which we will not treat in this work) and to the dimensional Brunn--Minkowski inequality \cite{cordero2025concavity, rot_B, dim_gauss, ChadGalynaL_p,Kol_Liv_Gardner-Zvavitch, GalynaDimensional_BM} (which we present later in \eqref{eq:dim_BM}). Both conjectures, already interesting on their own, have acquired further relevance in relation to the logarithmic Brunn--Minkowski conjecture, stated in \cite{log_in} by B\"or\"oczky, Lutwak, Yang, and Zhang, and are intimately connected to the logarithmic Minkowski problem, introduced by the same authors in \cite{log_Mink}. Thanks to the works of Saroglou \cite{Sar} and Livshyts, Marsiglietti, Nayar, and Zvavitch \cite{misc}, it is known that the logarithmic Brunn--Minkowski conjecture implies the dimensional Brunn--Minkowski conjecture and the (B)-conjecture. Further details on the development of these open problems can be found in \cite{cordero2025concavity}.

\medskip

Given a compact convex set $K\subset\R^n$ with non-empty interior and a measure $\mu$ with sufficiently regular density with respect to the Lebesgue measure, we denote by $H^1(\partial K,\mu)$ the weighted Sobolev space of real-valued functions on $\partial K$ with respect to $\mu$. The associated Sobolev norm is denoted by $\|\cdot\|_{H^1(\partial K,\mu)}$. Similarly, one defines $H^1(K,\mu)$ and the corresponding norm; see Section~3 for precise definitions. These spaces provide a natural extension for the choice of $\rho$ and $\varphi$ in Theorem \ref{thm:multiplicative_form}, where the two functions appear in a completely decoupled manner. Extending the validity of \eqref{eq:multiplicative_form} to the Sobolev spaces above justifies treating $\varphi$ as a parameter, allowing one to derive sharper inequalities for a fixed function $\rho\in H^1(\partial K,\mu)$. By exploiting a suitable Sobolev extension of $\rho$ to the whole set $K$, we obtain an interpolation-type inequality (see Proposition~\ref{prop:Boundary Sobolev form}) which, in turn, yields the following stability result for the weighted Poincar\'e inequality. 
\begin{theorem}\label{thm:stability_intro}
    Consider a non-empty compact convex set $K \subset \R^n$ such that $\partial K$ is a strictly convex manifold of class $C^2$. Let $u \in C^2(\R^n)$ be strictly convex and consider the measure $\mu$ with density $\di \mu(x) = e^{-u(x)}\di x$. If $\rho \in H^1(\partial K,\mu)$ is such that 
\begin{equation}\label{eq:stability_intro}
\int_{\partial K} \left\langle {\rm{II}}^{-1}\nabla_{\partial K}\rho, \nabla_{\partial K} \rho \right\rangle \di \mu \leq \int_{\partial K}\mathrm{H}_{\mu}\rho^2\di \mu  -\frac{1}{\mu(K)}\left( \int_{\partial K}\rho \di \mu  \right)^2 + \varepsilon
\end{equation}
for some $\varepsilon \in (0,1)$, then
\[\|\rho\|_{H^{1}(\partial K,\mu)} \leq C\varepsilon^{1/2}\]
for a constant $C$ depending only on $K$ and $\mu$.\\
In particular, equality is attained in \eqref{eq:Poincare} if and only if $\rho \equiv 0$.
\end{theorem}

\medskip

Let us now shortly present the dimensional Brunn--Minkowski conjecture. We say that $A \subset \R^n$ is origin-symmetric if $A=-A=\{-x \colon x \in A \}$, and that a measure $\mu$ is even if $\mu(B)=\mu(-B)$ for every $\mu$-measurable set $B \subset \R^n$. The question is the following: Given an even log-concave measure $\mu$ on $\R^n$, is it true that, for every pair of origin-symmetric non-empty compact convex sets $K,L \subset \R^n$, the inequality 
\begin{equation}\label{eq:dim_BM}
    \mu((1-t)K+tL)^{\frac{1}{n}}\geq (1-t)\mu(K)^{\frac{1}{n}}+t\mu(L)^{\frac{1}{n}}
\end{equation}
holds for every $t \in [0,1]$? In other words, does every even log-concave measure enjoy the same dimensional concavity as the Lebesgue measure when restricted to the class of origin-symmetric convex sets?

This conjecture was put forth by Gardner and Zvavitch for the Gaussian measure
\cite{gardner2010gaussian} and was resolved by Eskenazis and Moschidis \cite{dim_gauss}.
Cordero-Erausquin and Rotem \cite{rot_B} gave an affirmative answer for sufficiently
regular and rotationally invariant log-concave measures, while Kolesnikov and Livshyts \cite{Kol_Liv_Gardner-Zvavitch} established the conjecture for hereditarily
convex log-concave measures. For general even log-concave measures, Livshyts
\cite{GalynaDimensional_BM} proved \eqref{eq:dim_BM} with exponent $n^{-(4+o(1))}$,
which currently constitutes the best known universal bound. 

To study the dimensional Brunn--Minkowski conjecture, Livshyts \cite{GalynaStability} proposed the following local approach. Given a log-concave measure $\mu$ on $\mathbb{R}^n$ and a compact convex set $K \subset \mathbb{R}^n$ with non-empty interior, one seeks the maximal value $p(\mu,K)$ such that
\[\left. \frac{\di^2}{\di t^2} \mu\bigl(K(\rho,t)\bigr)^{p(\mu,K)} \right|_{t=0}\leq 0\quad \text{for every } \rho \in \mathcal{A},\]
where $K(\rho,t)$ denotes a perturbation of $K$ generated by a suitable function $\rho : \partial K \to \mathbb{R}$, and $\mathcal{A}$ is the class of admissible perturbations (see Section~4 for precise definitions). Building upon this idea, we shall reformulate the dimensional Brunn--Minkowski conjecture in a variational fashion, where the energy functional to be minimized turns out to depend crucially on the bilinear form $\langle \cdot, \cdot \rangle_{\rm P}$. With the main theorems we mentioned so far, as well as some preliminary estimates, we prove that not only is \eqref{eq:Poincare} saturated by trivial functions, but that $\langle \cdot, \cdot \rangle_{\rm P}$ is in addition coercive, therefore allowing the application of the Lax--Milgram theorem to $\langle \cdot, \cdot \rangle_{\rm P}$. From this procedure, we deduce the following result, which we formulate here as a statement on the existence and uniqueness of a weak solution to the corresponding Euler--Lagrange equation. Given a vector field
$X \in C^{1}(\partial K,\mathbb{R}^n)$, we denote by
$\nabla_{\partial K} \cdot X$ the tangential divergence of $X$ on $\partial K$.
\begin{theorem}\label{thm:minimal_power}
    Let $K$ and $\mu$ satisfy the assumptions of Theorem~\ref{thm:multiplicative_form}. Then there exists a unique weak solution $\overline{\rho} \in H^{1}(\partial K,\mu)$ to the equation
\begin{equation}\label{eq:PDE_divergence_form_intro}
-\nabla_{\partial K} \cdot ({\rm II}^{-1}\nabla_{\partial K}\rho)+\langle \nabla_{\partial K}u,{\rm II}^{-1}\nabla_{\partial K}{\rho} \rangle -{\rm H}_\mu\,\rho+\frac{1}{\mu(K)}\int_{\partial K}\rho\,\di\mu=1\qquad\text{on }\partial K,
\end{equation}
and it satisfies 
\begin{equation}\label{eq:new_concavity expression_intro}
p(\mu,K) = \frac{\mu(K)}{\int_{\partial K} \overline{\rho}\,\di\mu}.
\end{equation}
\end{theorem}\noindent
It is worth mentioning that the operator $-\nabla_{\partial K} \cdot ({\rm II}^{-1}\nabla_{\partial K} (\cdot) )+\langle \nabla_{\partial K}u,{\rm II}^{-1}\nabla_{\partial K}(\cdot) \rangle$ is nothing but a weighted Laplacian on $\partial K$ with an appropriate choice of metric and measure. Compare, e.g., \cite[Section 6.5]{Kol_Mil_P}. 

This equation paves the way to several reformulations of the conjecture~\eqref{eq:dim_BM}. For the moment, we highlight the following consequence (see Theorem \ref{thm:reformulation} later), which relies on the bilinear form $\langle \cdot, \cdot \rangle_{\rm I}$ introduced earlier. In particular, the dimensional Brunn--Minkowski conjecture is equivalent to the following local statement: For every origin-symmetric $K$ and even $u$ satisfying the assumptions of Theorem~\ref{thm:minimal_power}, the
inequality
\begin{equation}\label{eq:equivalent_form_intro}
\langle \overline{\rho}, \langle \nabla u, x \rangle \rangle_{\rm I}+\int_K \langle \nabla u, x \rangle\, \di \mu \geq 0
\end{equation}
holds, where $\overline{\rho}$ is the function provided by Theorem~\ref{thm:minimal_power}.  Indeed, $\overline{\rho}$ may be interpreted as an extremal perturbation of $K$, encoding the concavity properties of $K$ while simultaneously incorporating the geometry induced by the measure $\mu$. It is interesting to note that in the presence of symmetry, the minimizer is even. Moreover, we note that estimates similar to \eqref{eq:equivalent_form_intro} were already required in \cite{cordero2025concavity} (see Lemma~11 therein) to establish both the (B)-conjecture and the dimensional Brunn--Minkowski conjecture, together with certain spectral inequalities. By contrast, the Euler--Lagrange approach adopted here appears to bypass the need for such spectral inequalities, at the cost of a more delicate analysis of the PDE \eqref{eq:PDE_divergence_form_intro}.

\medskip

We conclude this manuscript with a result in the same spirit as \cite[Theorem~1.1]{Kol_Liv_Gardner-Zvavitch}, due to Kolesnikov and Livshyts, obtained under the same pinching conditions on the measure $\mu$. We denote the  identity matrix by $\Id$.
\begin{theorem}\label{thm:SG_intro}
Consider an origin-symmetric compact convex set $K \subset \R^n$ such that $\partial K$ is a strictly convex manifold of class $C^2$, and an even strictly convex function $u \in C^2(\R^n)$. Consider the measure $\mu$ with density $\di \mu(x) = e^{-u(x)}\di x$ and suppose, moreover, that $k_{1}\,\mathrm{Id} \leq \nabla^{2}u$ and $ \Delta u \leq k_{2}\,n$ for some $k_2\geq k_1 >0$.  Then
\[  p(\mu,K) \geq \frac{c}{n},\]
where $c=\frac{1}{r+1}$, with $r=\frac{k_{2}}{k_{1}}\geq 1$. \\
In particular, 
\[ \mu((1-t)K+tL)^{\frac{c}{n}}\geq (1-t)\mu(K)^{\frac{c}{n}}+t\mu(L)^{\frac{c}{n}}\] for every origin-symmetric non-empty compact convex sets $K, L \subset \R^n$ and $t \in [0,1]$.
\end{theorem}\noindent
Notice that the estimate obtained in \cite{Kol_Liv_Gardner-Zvavitch} is slightly sharper (even though morally equivalent up to an absolute constant). In the notation above, it corresponds to the constant $c=2/(\sqrt{r}+1)^2$. The bound derived here, however, is obtained through a substantially different procedure, building directly on the framework developed in the preceding sections. In particular, we believe this different method to be of independent interest. Finally, we note that, when $\mu$ is the Gaussian measure, Theorem~\ref{thm:SG_intro} implies the main result of \cite{Kol_Liv_Gardner-Zvavitch}; in particular, a close inspection of the proof implies that the symmetry assumption on the convex sets involved can be weakened to the assumption that the sets merely contain the origin.

\paragraph{Structure of the paper.} 

We start in Section 2 with an exposition of the perturbative methods necessary for our computations. In the same section, we prove Theorem \ref{thm:multiplicative_form}. In Section 3, we extend the bilinear forms $\langle\cdot,\cdot \rangle_{\rm P}$ and $\langle\cdot,\cdot \rangle_{\rm BL}$ to suitable Sobolev spaces, providing the machinery needed for Theorem \ref{thm:stability_intro}, and the proof of the theorem itself. We then start Section 4 proving that $\langle\cdot,\cdot \rangle_{\rm P}$ satisfies the assumptions of the Lax--Milgram theorem, which will imply Theorem \ref{thm:minimal_power}. In the same section, we discuss the corresponding Euler--Lagrange equation and some reformulations of the dimensional Brunn--Minkowski conjecture. The manuscript is concluded in Section 5 with the proof of Theorem \ref{thm:SG_intro}.

\section{A concavity principle along functional Wulff shapes}

For the convenience of the reader, we report first some preliminaries on convex sets and convex functions. For some complete introductions on these topics, the reader can consult the monographs by Schneider \cite{SchneiderConvexBodiesBrunn2013} and Rockafellar \cite{RockafellarConvex1997}. For $A, B \subset \R^n$, their Minkowski sum is the set \[ A+B=\{ x+y \in \R^n: x\in A, y \in B\}.\] We denote by $\K^n$ the family of compact convex subsets of $\R^n$ with non-empty interior, also known as convex bodies. This space is closed under Minkowski addition. If $K \in \K^n$, then it is uniquely determined by its support function $h_K(x)=\sup_{y \in K} \langle x, y \rangle, x \in \R^n$. By construction, support functions are positively $1$-homogeneous, and can thus be identified with their restrictions to the Euclidean unit sphere $\sph^{n-1}$ of $\R^n$. A convex body $K \in \K^n$ can be represented via its support function as \[ K=\{x \in \R^n : \langle x,\xi \rangle \leq h_K(\xi) \text{ for every }\xi \in \sph^{n-1} \}.\] This representation is connected with the so-called Wulff shape construction. Consider $f \in C(\sph^{n-1})$. Then, the corresponding Wulff shape is the set \[ [f]=\{x \in \R^n: \langle x , \xi \rangle \leq f(\xi) \text{ for every }\xi \in \sph^{n-1} \}.\] It is immediate to prove that the set thus obtained is always convex. Moreover, $[f] \in \K^n$ whenever $f>0$ up to the addition of a linear function.

Consider now the space of convex functions \[ \Conv(\R^n)=\{u: \R^n \to \R\cup \{\infty\}: u \text{ is convex, lower semi-continuous, and proper} \},\] where a function $u$ is proper if it is not identically equal to $\infty$. Equivalently, $\Conv(\R^n)$ is the space of functions $u: \R^n \to \R\cup\infty$ such that the corresponding epigraph $\epi(u)=\{(x,z) \in \R^n \times \R: u(x) \leq z \}$ is convex, closed, and non-empty.
On $\Conv(\R^n)$, there is an operation which is equivalent to the Minkowski addition. Given $u, v \in \Conv(\R^n)$, their infimal convolution is the function $u \square v$, which is characterized by the property $\epi(u \square v)=\epi(u)+\epi(v)$. Thus, $u \square v \in \Conv(\R^n)$. Alternatively, this operation can be explicitly written as \[ u \square v(x)= \inf_{y \in \R^n} \{ u(y)+v(x-y) \}.\] A further operation on the set of convex functions is the one of epi-multiplication. For $u \in \Conv(\R^n)$ and $t>0$, the function $t \sq u$ is described by the condition $\epi(t \sq u)=t\epi(u)$. Equivalently, \[t\sq u(x) = t u\left( \frac{x}{t} \right).\]
Observe that $\K^n$ can be identified inside $\Conv(\R^n)$ as the subset of indicator functions. That is, for $K \in \K^n$ we consider $I_K \in \Conv(\R^n)$ defined as 
\begin{align*}
     I_K(x)=\begin{cases}
        0  & \text{ if }x \in \K,\\
        \infty & \text{ otherwise.}
    \end{cases}
\end{align*}
Since $h_K \in \Conv(\R^n)$ for every $K \in \K^n$, support functions provide a further representation of convex bodies inside $\Conv(\R^n)$. We say that a function on $\sph^{n-1}$ is convex if its natural positive $1$-homogeneous extension is convex. Observe that if $f \in C(\sph^{n-1})$ is convex, then $f=h_{[f]}$.

We now introduce a fundamental tool for this section. Given a function $\psi: \R^n \to \R \cup \{\infty\}$, its Fenchel--Legendre transform is \[ \psi^*(x)=\sup_{y \in \R^n}\{ \langle x, y \rangle - \psi(y)\}.\] It is well known that if $u \in \Conv(\R^n)$ then $u^* \in \Conv(\R^n)$ and $(u^*)^*=u$. We will use several well-known properties of this transform. For more details, the reader can consult \cite[Section 26]{RockafellarConvex1997}. We recall, in particular, that for $f \in C(\sph^{n-1})$ one has the identity
\begin{equation}\label{eq:FL_Wulff}
    f^*=I_{[f]},
\end{equation}
where $f$ is identified with its $1$-homogeneous extension. As a special case, observe that $h_K^*=I_K$ for every $K \in \K^n$. Another useful property of the Fenchel--Legendre transform is the following: For $u,v \in \Conv(\R^n)$ and $t,s>0$,
\begin{equation}\label{eq:sums_FL}
    (tu^*+sv^*)^*=(t \sq u) \square (s \sq v).
\end{equation} 
In \cite{Ulivelli_Wulff}, the second-named author investigated the connections between Wulff shapes and the Fenchel--Legendre transform. In particular, the latter can be regarded as a functional version of the former. We will largely use this intuition in the remainder of this section.

\medskip

Theorem \ref{thm:multiplicative_form} will be obtained by differentiating the following concavity principle, which is a direct application of \cite[Theorem 6]{Prekopa}.
\begin{lemma}\label{lemma:concavity_principle_PL}
    Consider $u \in \Conv(\R^n)$ and $K \in \K^n$. Then, for every $\psi \in C(\R^n)$ and $f \in C(\sph^{n-1})$ the functional
    \begin{equation}\label{eq:conc_pr}
        \F(t)= \log \int_{[h_K + tf]} e^{-(u^*+t\psi)^*} \di x
    \end{equation}
    is concave in $t \in [0,1]$.
\end{lemma}
\begin{proof}
    Assume that $[h_K + tf]$ has non-empty interior and $(u^*+t\psi)^* \neq \infty$ on an open set for $t$ sufficiently small (rescaling, if necessary). Otherwise, $\F(t)=-\infty$ and the statement is trivial. 
    
    Observe that \[ \int_{[h_K + tf]} e^{-(u^*+t\psi)^*} \di x = \int_{\R^n} e^{-(u^*+t\psi)^* - I_{[h_K + tf]}} \di x. \] 
    As $(u^*+t\psi)^*(x)=\sup_{y \in \R^n} \{\langle x, y \rangle - u^*(y) - t \psi(y) \}$ and $\langle x, y \rangle - u^*(y) - t \psi(y)$ is convex in $(x,t)$ for every $y \in \R^n$, one readily checks that $(u^*+t\psi)^*(x)$ is jointly convex in $(x,t)$, as it is a supremum of convex functions. By \eqref{eq:FL_Wulff}, notice that $I_{[h_K+tf]}(x)=(h_K+tf)^*(x)$ and, analogously, this function is jointly convex in $(x,t)$. Therefore, $(u^*+t\psi)^*(x) + I_{[h_K + tf]}(x)$ is jointly convex and we can apply Pr\'ekopa's concavity principle \cite[Theorem 6]{Prekopa}, which entails our claim.
\end{proof}

As we will soon have to compute the first and second derivatives of \eqref{eq:conc_pr}, we need to assume further regularity on the sets and functions involved. We say that $K \in \K^n$ is a smooth convex body if $\partial K$ is a manifold of class $C^2$ and strictly convex in the sense described in the introduction, i.e.,
${\rm II} > 0$. In particular, $h_K$ is of class $C^2$ in $\R^n \setminus \{ o\}$ (compare \cite[Section 2.5]{SchneiderConvexBodiesBrunn2013}) and the main curvatures of $\partial K$ are strictly positive and bounded as functions on $\partial K$. Additionally, in this section, we will always tacitly assume that $o \in \mathrm{int}K$ (the interior of $K$), as the problems at hand will always be invariant under translation. We say that a function $u \in \Conv(\R^n)$ is smooth and strictly convex if $u \in C^2(\R^n)$ and $\nabla^2 u$ is positive definite. We now provide a statement concerning the computation of shape derivatives. In the context we are treating, similar facts have been proved in \cite{Kol_Mil_P} and \cite{cordero2025concavity}. For a function $\Phi_t$ differentiable in $t$ (and which might depend on other variables), we adopt the notation $\Phi'_t=\frac{\partial }{\partial t} \Phi_t$ for the partial derivative on the entry of $t$ only.
\begin{proposition}\label{prop:derivative}
    Consider a family of Lipschitz vector fields $X_t: \R^n \to \R^n, t \in (-\varepsilon,\varepsilon), \varepsilon>0$, such that $X_0=\Id$ and $X_t$ is differentiable and invertible for every $t \in (-\varepsilon,\varepsilon)$. Consider, moreover, a bounded set $A \subset \R^n$ with Lipschitz boundary, and define $A_t=X_t(A)$. If $F_t \in C^1(\R^n\times(-\varepsilon,\varepsilon))$. Then, for every $t \in (-\varepsilon,\varepsilon),$
        \begin{equation}\label{eq:first_derivative}
            \frac{\di}{\di t} \int_{A_t} F_t(x) \di x =\int_{A_t} F'_t (x) \di x + \int_{\partial A_t}F_t(x) \langle X'_t(X^{-1}_t(x)), \nu_{A_t}(x) \rangle \di \haus^{n-1}(x), 
        \end{equation}
        where $\nu_A$ is the outer unit normal on the boundary of $A$. In particular,
        \begin{equation}\label{eq:first_derivative_0}
            \left.  \frac{\di}{\di t} \int_{A_t} F_t(x) \di x \right|_{t=0} =\int_{A} F'_0 (x) \di x + \int_{\partial A}F_0(x) \langle X'_0(x), \nu_{A}(x) \rangle \di \haus^{n-1}(x).
        \end{equation}
\end{proposition}
\begin{proof}
    Formulas analogous to \eqref{eq:first_derivative} can be found, under similar regularity assumptions, in \cite[Section 5]{Henrot_Shape}. We provide an autonomous proof for the convenience of the reader and to present the result in a fashion which is better suited for our applications.

    The proof of \eqref{eq:first_derivative} follows by the area formula \cite[Theorem 3.2.3]{The_Bible} and the divergence theorem. Indeed, by the area formula, \[ \int_{A_t} F_t(x) \di x = \int_{A} F_t(X_t(x))\det(\nabla X_t(x)) \di x,\] where no multiplicities appear since $X_t$ is invertible. The Jacobian $\nabla X_t$ is understood as an approximate Jacobian, which exists almost everywhere since $X_t$ is Lipschitz. Next, we differentiate with respect to $t$ under the integral sign. First, observe that, as $X_0=\Id$, $\nabla X_t(x)$ is invertible in a neighbourhood of $0$, which we can suppose to be $(-\varepsilon,\varepsilon)$ choosing a smaller $\varepsilon>0$ if needed. Then, by Jacobi's formula, \[ \frac{\di}{\di t}\det(\nabla X_t(x)) = \det(\nabla X_t(x)) \tra \left( (\nabla X_t (x))^{-1} \nabla X'_t(x)\right).\] Applying the chain rule to derivate $F_t(X_t(x))$ and by changing variable back to $A_t$ we infer 
    \begin{align*}
        & \frac{\di}{\di t} \int_{A_t} F_t(x) \di x  =  \int_A F'_t(X_t(x))\det(\nabla X_t (x)) \di x \\&+\int_A \left(\langle \nabla F_t(X_t(x)), X'_t(x)\rangle + F_t(X_t(x)) \tra \left( (\nabla X_t (x))^{-1} X'_t(x)\right) \right)\det(\nabla X_t (x)) \di x\\
        & = \int_{A_t} F'_t(x) \di x +\int_{A_t} \langle \nabla F_t(x), X'_t(X^{-1}_t(x))\rangle + F_t(x) \tra \left( (\nabla X_t (X^{-1}_t(x)))^{-1} X'_t(X^{-1}_t(x))\right) \di x.
    \end{align*}
    Observe now that $\nabla  X'_t(X^{-1}_t(x))=(\nabla X_t (X^{-1}_t(x)))^{-1} X'_t(X^{-1}_t(x))$ since, by computing the gradient on both sides of the identity $x=X_t \circ X^{-1}_t(x)$, it follows that $\nabla X^{-1}_t(x)=\left( \nabla X_t(X^{-1}_t (x))\right)^{-1}$. Therefore, 
    \begin{align*}
        & \int_{A_t} \langle \nabla F_t(x), X'_t(X^{-1}_t(x))\rangle + F_t(x) \tra \left( (\nabla X_t (X^{-1}_t(x)))^{-1} X'_t(X^{-1}_t(x))\right) \di x\\
        =& \int_{A_t} \nabla \cdot \left( F_t(x) X'_t(X^{-1}_t(x))\right) \di x = \int_{\partial A_t} F_t(x) \langle X'_t(X^{-1}_t(x)),\nu_{K_t}(x) \rangle \di \haus^{n-1},
    \end{align*}
    where in the last line we have applied the divergence theorem, thus proving \eqref{eq:first_derivative}. Finally, \eqref{eq:first_derivative_0} follows by direct substitution, using the fact that $X_0=\Id$.
\end{proof}

We now introduce our choice for $X_t$, which can be compared with the flow studied in \cite[Section 6.2.3]{Kol_Mil_P}. Consider a smooth convex body $K \subset \R^n$ such that $o \in {\rm int} K$ and $f \in C^2(\sph^{n-1})$. Then, for $t \in (-\varepsilon,\varepsilon)$ for $\varepsilon>0$ sufficiently small, the Wulff shape $K_t=[h_K+tf]$ is a smooth convex body such that $h_{K_t}=h_K+tf$ and $o \in {\rm int}K_t$. We first define $X_t$ on $\partial K$ as
\begin{align*}
    X_t : \partial K &\to \partial K_t \\
    x &\mapsto \nu^{-1}_{K_t}\circ \nu_K(x),
\end{align*}
where $t \in (-\varepsilon,\varepsilon)$. By a well-known property of support functions (compare, e.g., \cite[Corollary 1.7.3]{SchneiderConvexBodiesBrunn2013}) and by the representation of $h_{K_t}$, for every $\xi \in \sph^{n-1}$, \[ \nu^{-1}_{K_t}(\xi)= \nabla h_{K_t}(\xi)=\nabla (h_K+tf)=\nu^{-1}_K(\xi)+t\nabla f(\xi).\] Choosing $\xi=\nu_K(x)$ for $x \in \partial K,$ we infer \[ X_t(x)=x+t \nabla f(\nu_K(x)).\] 
With the next lemma, we extend $X_t$ to $\R^n$. For $K \in \K^n$ such that $o \in {\rm int}K$ we denote by \[ ||x||_K=\inf\{\tau \geq 0 : x \in \tau K \}, \, x \in \R^n,\] its gauge function (compare \cite[Section 1.7]{SchneiderConvexBodiesBrunn2013}).
\begin{lemma}\label{lemma:malaka_vector_field}
    Under the previous assumptions on $K$ and $f$, the vector field 
    \begin{align*}
        X_t: \R^n &\to \R^n \label{eq:vector_field}\\
    x &\mapsto x+ t||x||_K\nabla f\left( \nu_K\left( \frac{x}{||x||_K} \right) \right). 
    \notag
    \end{align*}
    satisfies the conditions of Proposition \ref{prop:derivative}. Moreover, $X_t(K)=K_t=[h_K+tf]$.
\end{lemma}
\begin{proof}
    For every $\tau \geq 0$ we define 
    \begin{align*}
    X^\tau_t : \partial (\tau K) &\to \partial (\tau K_t) \\
    x &\mapsto \nu^{-1}_{\tau K_t}\circ \nu_{\tau K}(x).
\end{align*}
Exploiting the properties of the support functions (in particular, $h_{\tau K}=\tau h_K$ for every $\tau \geq 0$), for every $x \in \partial (\tau K)$ 
\begin{align*}
    X^\tau_t(x)=\nu^{-1}_{\tau K_t} \circ \nu_{\tau K}(x)= \nabla h_{\tau K_t}(\nu_{\tau K}(x))=\nabla(h_{\tau K}+\tau t f)(\nu_{\tau K}(x))=x+\tau t \nabla f(\nu_{\tau K}(x)).
\end{align*}
From the properties of the gauge function, $x \in \partial (||x||_K K)$ for every $x \in \R^n$. Therefore, if $x \in \partial \tau K$ then $\tau = ||x||_K$ and $\nu_{\tau K}(x)=\nu_K(x/||x||_K)$. Under the imposed regularity assumptions, if we now define 
\[ X_t(x)=x+t ||x||_K \nabla f \left( \nu_K\left( \frac{x}{||x||_K}\right)\right), \] this vector field is clearly well-defined on $\R^n$, Lipschitz, differentiable for every $t \in (-\varepsilon,\varepsilon)$, and $X_0=\Id$. Finally, to check that $X_t$ is invertible for every $t \in (-\varepsilon,\varepsilon)$, observe that, for every $\tau\geq 0$ and every fixed $t \in (-\varepsilon,\varepsilon),$ $X_t$ is, by construction, a bijection (in fact, a diffeomorphism) between $\partial (\tau K)$ and $\partial (\tau K_t)$. Since both $K$ and $K_t$ have the origin inside their interior, $\R^n$ can be represented as the disjoint unions \[ \dot{\bigcup}_{\tau \geq 0} \partial (\tau K) = \R^n = \dot{\bigcup}_{\tau \geq 0}( \partial (\tau K_t)),\] proving that $X_t$ is a bijection on $\R^n$ for every $t \in (-\varepsilon,\varepsilon)$. Analogously, one checks that $X_t(K)=K_t$, concluding the proof.
\end{proof}

The next step is to establish hypotheses under which the integrand in \eqref{eq:conc_pr} is sufficiently regular to apply Proposition \ref{prop:derivative}. Consider a smooth and strictly convex function $u$. Additionally, assume that there exist $k_2 \geq k_1 >0$ such that $k_1 \Id \leq \nabla^2 u \leq k_2 \Id$. Equivalently, all the eigenvalues of $\nabla^2 u$ lie between $k_1$ and $k_2$. By the properties of the Fenchel--Legendre transform, this implies that $k_2^{-1} \Id \leq \nabla^2 u^* \leq k_1^{-1} \Id$. Therefore, if we choose $\psi \in C^2(\R^n)$ such that its second derivatives are uniformly bounded on $\R^n$, the functions $u^*+t\psi$ and $(u^*+t\psi)^*$ are both smooth and strictly convex for $t$ sufficiently small. We recall the following result, which can be considered folklore under our regularity assumptions. See, for example, \cite[Propositions 5.1 and 5.3]{Artstein_diff_pol}.
\begin{lemma}\label{lemma:first_second}
    Consider $u_t=(u^*+t\psi)^*$ with $u,\psi$ as above. Then, for every $t$ sufficiently small, \[ u'_t(x) =-\psi(\nabla u_t(x)) \quad \text{ and } \quad  u''_t(x)= - \left\langle \nabla^2 u_t(x)   \nabla \psi (\nabla u_t(x)) , \nabla \psi (\nabla u_t(x)) \right \rangle\] for every $x \in \R^n$.  
\end{lemma}\noindent
We now have all the ingredients for explicitly computing the derivatives we need. In the following, we omit the dependence on $x$ when it is clear from the context to simplify the notation. 
\begin{proposition}\label{prop:explicit_derivatives}
    Consider a smooth convex body $K \subset \R^n$ and a function $u\in C^2(\R^n)$ such that there exist $k_2 \geq k_1 >0$ with $k_1 \Id \leq \nabla^2 u \leq k_2 \Id$. For $f \in C^2(\sph^{n-1})$ and $\psi \in C^2(\R^n)$ with bounded second derivatives consider the functional \[ I(t)=\int_{[h_K+tf]} e^{-(u^*+t\psi)^*}\di x.\] Then
    \begin{equation}\label{eq:first_explicit}
        I'(0)=\int_{K} \psi(\nabla u)e^{-u}\di x + \int_{\partial K} f(\nu_K)e^{-u}\di \haus^{n-1}
    \end{equation}
    and
    \begin{align}\label{eq:second_explicit}
        I''(0)  = & \int_K \psi(\nabla u)^2 e^{-u} \di x -\int_K \langle(\nabla^2 u)^{-1}\nabla \left(\psi(\nabla u)\right),\nabla \left(\psi(\nabla u)\right) \rangle e^{-u}\di x  \notag  \\  + & 2\int_{\partial K} f(\nu_K)\psi(\nabla u)e^{-u}\di \haus^{n-1}+ 
        \int_{\partial K}(\tra({\rm II})-\langle \nabla u, \nu_K \rangle)f(\nu_K)^2 e^{-u}\di \haus^{n-1}\\ - & \int_{\partial K} -\langle {\rm II}^{-1}\nabla_{\partial K} (f(\nu_K)),\nabla_{\partial K} (f(\nu_K))\rangle e^{-u}\di \haus^{n-1}. \notag
    \end{align}
\end{proposition}
\begin{proof}
       By the regularity assumptions above, there exists $\varepsilon>0$ such that, for every $t \in (-\varepsilon,\varepsilon)$ $K_t=[h_K+tf]$ is a smooth and strictly convex body and $u_t=(u^*+t\psi)^* \in C^2(\R^n)$ is convex and the eigenvalues of its Hessian matrix are bounded from below and above by strictly positive constants (see the discussion before Lemma \ref{lemma:first_second}). In particular, $I(t)$ is always well-defined and finite. By construction, $K_t=X_t(K)$, where $X_t$ is chosen as in Lemma \ref{lemma:malaka_vector_field} and we can suppose, without loss of generality, since $I(t)$ is invariant under simultaneous translation of integrand and domain of integration, that $o \in {\rm int}K$. Additionally, by Lemma \ref{lemma:first_second}, $F_t=e^{-{u_t}}$ satisfies the assumptions of Proposition \ref{prop:derivative} and, therefore, we can apply \eqref{eq:first_derivative} to compute $I'(t)$ with the further choice $A=K$. By Lemma \ref{lemma:first_second} \[ F'_t(x)=\varphi(\nabla u_t(x))e^{-u_t(x)}.\] Moreover,
       \begin{align*}
           X'_t(x)= &||x||_K\nabla f(\nu_K(x/||x||_K))\\ =&||x||_K\left( \nabla_{\sph^{n-1}} f(\nu_K(x/||x||_K))+f(\nu_K(x/||x||_K)\nu_K(x/||x||_K) \right),
       \end{align*}
       where we have used that, since $f$ is $1$-homogeneous, $\nabla f (\xi)=\nabla_{\sph^{n-1}}f(\xi)+f(\xi)\xi$ when $\xi \in \sph^{n-1}$, and $\nabla_{\sph^{n-1}}$ denotes the covariant derivative on $\sph^{n-1}$. Notice that, under our regularity assumptions, $F'_t$ and $X'_t$ are continuous for $t \in (-\varepsilon,\varepsilon)$, which implies that we can compute $I'(t)$ in the whole interval, obtaining by \eqref{eq:first_derivative},
    \begin{equation}\label{eq:first_d_explicit}
        I'(t)=\int_{K_t} \psi(\nabla u_t) e^{-u_t} \di x+\int_{\partial K_t} e^{-u_t} \langle X'_t(X^{-1}_t),\nu_{K_t} \rangle \di \haus^{n-1}.
    \end{equation}
    In particular, substituting $t=0$ we obtain \eqref{eq:first_explicit}.

    We now compute $I''(0)$. First, we apply \eqref{eq:first_derivative} to the first integral on the right-hand side of \eqref{eq:first_d_explicit} with the choice $F_t=\psi(\nabla u_t)e^{-u_t}$, where the differentiability in $t$ is guaranteed by Lemma \ref{lemma:first_second}. By the explicit derivatives from Lemma \ref{lemma:first_second} we infer 
    \begin{align} \label{eq:explicit_00}
        \left. \frac{\di}{\di t}\int_{K_t} \psi(\nabla u_t)e^{-u_t}\di x \right|_{t=0}=&\int_K \psi(\nabla u)^2 e^{-u} \di x + \int_{\partial K} f(\nu_K)\psi(\nabla u)e^{-u}\di \haus^{n-1} \\ -&\int_K \langle(\nabla^2 u)^{-1}\nabla \left(\psi(\nabla u)\right),\nabla \left(\psi(\nabla u)\right) \rangle e^{-u}\di x, \notag
    \end{align}
    where we have used the identity \[ \langle(\nabla^2 u)\nabla \psi(\nabla u),\nabla \psi(\nabla u) \rangle=\langle(\nabla^2 u)^{-1}\nabla \left(\psi(\nabla u)\right),\nabla \left(\psi(\nabla u)\right) \rangle.\]
    Concerning the boundary integral in \eqref{eq:first_d_explicit}, observe that, by change of variable,  
    \begin{align} \notag
         &\int_{\partial K_t} e^{-u_t} \langle X'_t(X^{-1}_t),\nu_{K_t} \rangle \di \haus^{n-1}  \\ \label{eq:explicit_0}= & \int_{\partial K}e^{-u_t(X_t)}\langle \nabla f (\nu_K),\nu_K\rangle \det( \nabla_{\partial K} X_t)\di \haus^{n-1}\\ = & \int_{\partial K} f(\nu_K)e^{-u_t(X_t)} \det( \nabla_{\partial K} X_t) \di \haus^{n-1}, \notag
    \end{align}
    where $\nabla_{\partial K} X_t$ is the tangential Jacobian of $X_t$ (see, for example, \cite[Definition 5.4.2]{Henrot_Shape}), which in our case, can be computed explicitly as
    \begin{equation}\label{eq:tan_Jac}
         \nabla_{\partial K} X_t = \Id + t {\rm II}\left(\nabla_{\sph^{n-1}}^2 f (\nu_K) + f(\nu_K)\Id\right).
    \end{equation}
    We now compute the derivatives in $t$ of the integrands in \eqref{eq:explicit_0}. By \eqref{eq:tan_Jac} and Jacobi's rule,
    \begin{equation}\label{eq:explicit_1}
         \left. \frac{\di}{\di t}  \det( \nabla_{\partial K} X_t) \right|_{t=0}= \tra\left(  \nabla_{\sph^{n-1}}^2 f (\nu_K){\rm II}\right)+f(\nu_K)\tra( \rm II).
    \end{equation}
    Continuing our computations, by the chain rule, 
    \begin{align}\label{eq:explicit_2}
        \left. \frac{\di}{ \di t} e^{-u_t(X_t)}\right|_{t=0} = & \psi(\nabla u)e^{-u}-\langle \nabla u,  \nabla f(\nu_K)\rangle e^{-u} \\ = & \psi(\nabla u)e^{-u}-\langle \nabla u, \nu_K\rangle f(\nu_K) e^{-u}-\langle \nabla_{\partial K} u, \nabla_{\sph^{n-1}} f(\nu_K) \rangle e^{-u}. \notag
    \end{align}
    Now, observe that
    \begin{align} \label{eq:explicit_3}
        -&\langle \nabla_{\partial K} u, \nabla_{\sph^{n-1}} f(\nu_K) \rangle f(\nu_K) e^{-u}=\nabla_{\partial K} \cdot \left(e^{-u}f(\nu_K)\nabla_{\sph^{n-1}} f(\nu_K) \right)\\- &\langle \nabla_{\sph^{n-1}}f(\nu_K),\nabla_{\partial K}(f(\nu_K))\rangle -  \tra( \nabla_{\sph^{n-1}}^2 f (\nu_K){\rm II}). \notag
    \end{align}
    By \eqref{eq:explicit_1}, \eqref{eq:explicit_2}, and \eqref{eq:explicit_3}, 
    \begin{align}
       & \left. \frac{\di}{\di t}  f(\nu_K)e^{-u_t(X_t)} \det( \nabla_{\partial K} X_t) \right|_{t=0} \notag \\ \label{eq:explicit_4}
        = & f(\nu_K)\psi(\nabla u) e^{-u} -\langle {\rm II}^{-1}\nabla_{\partial K} (f(\nu_K)),\nabla_{\partial K} (f(\nu_K))\rangle\\+&(\tra({\rm II})-\langle \nabla u, \nu_K \rangle)f(\nu_K)^2 +\nabla_{\partial K} \cdot \left(e^{-u}f(\nu_K)\nabla_{\sph^{n-1}} f(\nu_K) \right), \notag
    \end{align}
    where we have used that, by the chain rule, $\nabla_{\partial K} (f(\nu_K))= \nabla_{\sph^{n-1}}f(\nu_K)  {\rm II}$. Plugging \eqref{eq:explicit_4} in \eqref{eq:explicit_0} and by the divergence theorem (recall that $\partial K_{t}$ has no boundary!), we finally infer 
    \begin{align} \label{eq:explicit_5}
        &\left. \int_{\partial K_t} e^{-u_t} \langle X'_t(X^{-1}_t),\nu_{K_t} \rangle \di \haus^{n-1} \right|_{t=0}=  \int_{\partial K}  f(\nu_K)\psi(\nabla u) e^{-u} \di \haus^{n-1}
        \\ + &\int_{\partial K}(\tra({\rm II})- \langle \nabla u, \nu_K \rangle)f(\nu_K)^2- \langle {\rm II}^{-1}\nabla_{\partial K} (f(\nu_K)),\nabla_{\partial K} (f(\nu_K))\rangle \di \haus^{n-1}. \notag
    \end{align}
    The proof is concluded observing that $I''(0)=\eqref{eq:explicit_00}+\eqref{eq:explicit_5}$, thus proving \eqref{eq:second_explicit}.
\end{proof}

We now have the following.
\begin{theorem}\label{thm:mean_form}
    Consider a smooth and strictly convex body $K \subset \R^n$, and a smooth and strictly convex function $u: \R^n \to \R$. Denote by $\mu$ the log-concave measure with density $\di \mu(x)=e^{-u(x)}\di x$. Then, for every Lipschitz function $\rho : \partial K \to \R$ and locally Lipschitz function $\varphi: \R^n \to \R$, 
    \begin{equation}\label{eq:mean_form}
        \langle \rho, \varphi\rangle_{\rm I} \leq \frac{\langle \rho, \rho \rangle_{\rm P}+\langle \varphi, \varphi \rangle_{\rm BL}}{2},
    \end{equation}
    where the bilinear forms in \eqref{eq:mean_form} are those defined in \eqref{eq:I},\eqref{eq:P}, and \eqref{eq:BL}, respectively.
\end{theorem}
\begin{proof}
     As a reduction step, consider $u,\psi$, and $f$ as in Proposition \ref{prop:explicit_derivatives}. The final result will then follow by standard approximation results and the dominated convergence theorem. Using the previous notation, we can write \[\F(t)=\log I(t)=\log \int_{K_t}e^{-u_t} \di x,\] which is concave by Lemma \ref{lemma:concavity_principle_PL}. If $\F(t)$ is twice differentiable at $0$, its concavity is equivalent to the inequality
     \begin{equation}\label{eq:concavity}
         I''(0)-\frac{I'(0)^2}{I(0)} \leq 0,
     \end{equation}
     where $I(0)=\mu(K)\neq 0$. By Proposition \ref{prop:explicit_derivatives}, $\F(t)$ is indeed twice differentiable at $0$, and we write \eqref{eq:concavity} as 
     \begin{align*}
         & 0 \geq \int_K \psi(\nabla u)^2\di \mu  -\int_K \langle(\nabla^2 u)^{-1}\nabla \left(\psi(\nabla u)\right),\nabla \left(\psi(\nabla u)\right) \rangle \di\mu + 2\int_{\partial K} f(\nu_K)\psi(\nabla u)\di \mu \\ + &
        \int_{\partial K}(\tra({\rm II})-\langle \nabla u, \nu_K \rangle)f(\nu_K)^2 e^{-u}-\langle {\rm II}^{-1}\nabla_{\partial K} (f(\nu_K)),\nabla_{\partial K} (f(\nu_K))\rangle e^{-u}\di \mu\\
        -& \frac{1}{\mu(K)}\left(\int_K \psi(\nabla u) \di \mu \right)^2 - \frac{1}{\mu(K)}\left(\int_{\partial K} f(\nu_K) \di \mu \right)^2 -\frac{2}{\mu(K)}\left(\int_K \psi(\nabla u) \di \mu \right)\left(\int_{\partial K} f(\nu_K) \di \mu \right).
     \end{align*}
     Notice now that any two functions $\rho : \partial K \rightarrow \mathbb{R}$ and $\varphi: K \rightarrow \mathbb{R}$ can be written as $\rho=f \circ \nu_K$ and $\varphi=\psi \circ \nabla u$ for some functions $f:\sph^{n-1}\rightarrow \mathbb{R}$ and $\psi: \nabla u(K) \rightarrow \mathbb{R}$, under our assumptions. Also, again by the assumptions on $K$ and $u$, $\nu_K$, and $\nabla u$ preserve the Lipschitz and locally Lipschitz conditions respectively and \eqref{eq:mean_form} is obtained rewriting the inequality above with the notations introduced in \eqref{eq:P}, \eqref{eq:BL}, and \eqref{eq:I}, completing the proof.
\end{proof}

Notice that, for the trivial choice $\rho\equiv 0$, \eqref{eq:mean_form} yields \eqref{eq:Brascamp_Lieb}. Analogously, $\varphi \equiv 0$ in \eqref{eq:mean_form} yields \eqref{eq:Poincare}. To finish this section, we rewrite Theorem \ref{thm:mean_form} in the form of Theorem \ref{thm:multiplicative_form}. To do so, we need the following elementary lemma.
\begin{lemma}\label{lemma:bilinear}
    Consider two real vector spaces $X, Y$, some positive semidefinite bilinear forms $\langle \cdot,\cdot\rangle_a : X \times X \to \R$, $\langle \cdot,\cdot\rangle_b : Y \times Y \to \R$, and a bilinear form $\langle \cdot,\cdot\rangle_c : X \times Y \to \R$. Then,
    \begin{equation}\label{eq:bilinear_av}
        \langle x, y\rangle_c \leq \frac{\langle x,x\rangle_a+\langle y,y\rangle_b}{2} \text{  for every } x \in X, y \in Y,
    \end{equation}
    if and only if
    \begin{equation}\label{eq:bilinear_prod}
        \langle x, y\rangle^2_c \leq \langle x,x\rangle_a\langle y,y\rangle_b \text{  for every } x \in X, y \in Y.
    \end{equation}
    Moreover, equality in \eqref{eq:bilinear_av} implies equality in \eqref{eq:bilinear_prod}.
\end{lemma}
\begin{proof}
    If \eqref{eq:bilinear_av} is true, we may substitute $tx$ in place of $x$. In particular, we obtain \[ 0\leq t^2\langle x,x\rangle_a- 2t\langle x, y\rangle_c +\langle y,y\rangle_b\] for every $x \in X, y \in Y, t \in \mathbb{R}$. Suppose that $\langle x,x\rangle_a\neq 0$. Substituting the value $t=\langle x,y \rangle_c/\langle x, x\rangle_a$, which minimizes the right-hand side, we infer \eqref{eq:bilinear_prod} in this case. The analogous procedure works if $\langle y,y\rangle_b \neq 0$ instead. If $\langle x,x\rangle_a=\langle y,y\rangle_b=0$, then \eqref{eq:bilinear_av} states that $\langle x, y\rangle_c \leq 0$. Replacing $x$ with $-x$, we similarly obtain that $\langle x, y\rangle_c \geq 0$. Therefore, $\langle x, y\rangle_c = 0$ and \eqref{eq:bilinear_prod} holds trivially. The converse implication is an immediate application of the Arithmetic/Geometric-mean inequality.

    To conclude the proof, observe that if equality holds in \eqref{eq:bilinear_av}, again by the inequality between Arithmetic and Geometric means, 
    \[ \langle x, y \rangle_c = \frac{\langle x,x\rangle_a+\langle y,y\rangle_b}{2} \geq \sqrt{\langle x,x\rangle_a\langle y,y\rangle_b}  \geq \langle x, y \rangle_c, \] and equality holds in \eqref{eq:bilinear_prod}.
\end{proof}\noindent
\begin{proof}[Proof of Theorem \ref{thm:multiplicative_form}.] This is just an immediate application of Lemma \ref{lemma:bilinear} to Theorem \ref{thm:mean_form}, where $X$ is the space of Lipschitz functions on $\partial K$ with the product $\langle\cdot,\cdot\rangle_{\rm P}$, $Y$ is the space of locally Lipschitz functions with the product $\langle \cdot, \cdot \rangle_{\rm BL}$, and the bilinear form between $X$ and $Y$ is $\langle\cdot,\cdot\rangle_{\rm I}$.
\end{proof}

We conclude this section with a remark on the equality cases of Theorem \ref{thm:multiplicative_form}. As anticipated in Theorem \ref{thm:stability_intro}, $\langle \rho, \rho \rangle_{\rm P}=0$ if and only if $\rho\equiv 0$. Moreover, Livshyts \cite{GalynaStability} has proved that the equality cases of the Brascamp--Lieb inequality are saturated by constants when the inequality is restricted to compact sets. We shall now show that our new infinitesimal forms \eqref{eq:mean_form} and \eqref{eq:multiplicative_form} of the Pr\'ekopa--Leindler inequality \eqref{eq:PL_1} admit further non-trivial equality cases. Observe that, by Lemma \ref{lemma:bilinear}, equality in \eqref{eq:mean_form} implies equality in \eqref{eq:multiplicative_form}, but not vice-versa. That is, \eqref{eq:multiplicative_form} is sharper than \eqref{eq:mean_form}. A full characterization is out of the scope of this paper, and will be the subject of future research.
\begin{proposition}\label{prop:equality}
    Consider $K$ and $\mu$ as in Theorem \ref{thm:multiplicative_form}. Then, for every $\alpha \geq 0, z \in \R,$ and $x_0 \in \R^n$, the functions $\rho=\alpha h_{K+x_0}(\nu_K)$ and $\varphi(x)=\alpha u^*(\nabla u(x-x_0))+z$ give equality in \eqref{eq:mean_form} and, therefore, in \eqref{eq:multiplicative_form}. 
\end{proposition}
\begin{proof}
    From the equality cases of the Pr\'ekopa--Leindler inequality it is well known that, for every $w, v \in \Conv(\R^n)$ such that $e^{-w}$ and $e^{-v}$ are integrable,
    \begin{equation}\label{eq:equality_cases}
        \int_{\R^n} e^{-((1-t)\sq w) \square (t\sq v)} \di x = \left( \int_{\R^n} e^{-w} \di x\right)^{(1-t)}\left( \int_{\R^n} e^{-v} \di x\right)^t
    \end{equation}
    if and only if the epigraphs of $w$ and $v$ are homothetic. That is, there exist $\alpha \geq 0, z \in \R,$ and $x_0 \in \R^n$ such that $v(x)=\alpha w(x-x_0)+z$. 
    Suppose now that, in Theorem \ref{thm:mean_form}, $\rho=\alpha h_K(\nu_K)$ and $\varphi=\alpha u^*(\nabla u)$, In particular, $\psi=\alpha u^*$ in the concavity principle of \eqref{eq:conc_pr}. Recall that \eqref{eq:mean_form} is obtained from differentiating twice \[ \F(t)=\log \int_{\R^n}e^{-(u^*+t\alpha u^*)^*-I_{(1+t\alpha)K}}\di x.\] By the properties of the Fenchel--Legendre transform and \eqref{eq:sums_FL}, \[ (u^*+t\alpha u^*)^*+I_{(1+t\alpha)K}=((1-t\alpha)\sq u)\square (t\alpha \sq u).\]In particular, by \eqref{eq:equality_cases}, $\F(t)$ is linear and, therefore, $\F''(0)=0$, which implies equality in \eqref{eq:mean_form}. In turn, by Lemma \ref{lemma:bilinear}, equality is attained in \eqref{eq:multiplicative_form}. The complete statement is obtained considering arbitrary translations of $h_K$ and $u^*$. 
\end{proof}

\section{Stability for the Poincar\'e inequality}

The aim of this section is to prove Theorem \ref{thm:stability_intro}. We point out that both the stability and the equality cases of \eqref{eq:Poincare} could be alternatively obtained by a careful inspection of its proof in \cite{Kol_Mil_P} via the $L_2$ method. Here, instead, we propose a different approach based on Theorem \ref{thm:multiplicative_form}. Further applications will be shown in the remaining sections.

\medskip

As a first step, we want to extend the validity of Theorem \ref{thm:multiplicative_form} to encompass two natural Sobolev spaces as the test functions in the inequality \eqref{eq:multiplicative_form}. For the convenience of the reader, we quickly recap some basics from the theory of (weighted) Sobolev spaces. Since the admissible measures $\mu$ in Theorem~\ref{thm:multiplicative_form} are absolutely continuous with respect to the Lebesgue measure restricted to $K$, and all sets involved are compact and sufficiently regular, the standard results from the classical theory of Sobolev spaces extend to our setting with only minor modifications. We refrain from presenting the full theory of traces, limiting the exposition to the essential material needed below. For complete references, see, for instance, \cite{SobBook,BrezisBook,LeoniBook}.

Let $K \subset \mathbb{R}^n$ be a smooth convex body and $u:\R^n \to \R$ a smooth and strictly convex function, and consider on $K$ the log-concave measure $\mu$ with density $\di \mu(x)=e^{-u(x)}\di x$. For $\psi \in C^1(K)$, consider the weighted Sobolev norm
\[
\|\psi\|_{H^{1}(K,\mu)}^{2}
  = \int_{K} \left(\psi^{2} + \|\nabla\psi\|^{2}\right)\, \di\mu,
\]
where $\|x \|^2=\langle x, x \rangle$ for every $x \in \R^n$. The weighted Sobolev space $H^{1}(K,\mu)$ is defined as the completion of $C^{\infty}(K)$ under $\|\cdot\|_{H^{1}(K,\mu)}$. On $\partial K$ we set, for $\rho \in C^{1}(\partial K)$,
\[
\|\rho\|_{H^{1}(\partial K,\mu)}^{2}
   = \int_{\partial K}\left(\rho^{2} + \|\nabla_{\partial K}\rho\|^{2}\right)\, \di\mu.
\]
Again for $\rho \in C^{1}(\partial K)$,  we denote its fractional Sobolev norm by
\begin{equation}\label{eq:Fractional Sobolev norm}
\|\rho\|^2_{H^{1/2}(\partial K,\mu)}
   = \int_{\partial K}\rho^2 \di \mu+\int_{\partial K}\!\int_{\partial K}
    \frac{|\rho(x)-\rho(y)|^{2}}{|x-y|^{n}}
    \di \mu(x) \di\mu(y).
\end{equation}
The spaces $H^{1}(\partial K,\mu)$ and $H^{1/2}(\partial K,\mu)$ correspond to the completions of the space $C^{\infty}(\partial K)$ with the respective norms. 

For a function $\psi \in C^{1}(K)\cap H^{1}(K,\mu)$, we can consider the trace operator, which is defined on $C^{1}(K)$ by
\[Tr(\psi) = \psi |_{\partial K}, \,\,\, \forall \psi \in C^{1}(K),\]
which extends to $H^{1}(K,\mu)$ by continuity. Given now a function $\rho \in H^{1/2}(\partial K,\mu)$, classical results in Sobolev extension theory, see, e.g., \cite[Th. 18.40]{LeoniBook}, guarantee the existence of a constant $C_{1}=C_{1}(K,\mu)$ and a function $\tilde{\psi} \in H^{1}(K,\mu)$ such that $Tr\left(\tilde{\psi}\right) = \rho $ and
\begin{equation}\label{eq:fractional_extension}
\|\tilde{\psi}\|_{H^{1}(K,\mu)} \leq C_{1} \,\|\rho\|_{H^{1/2}(\partial K,\mu)}.     
\end{equation}
Recall also that the space $H^{1}(\partial K,\mu)$, under our assumptions on $\mu$ and $K$, can be continuously embedded into  $H^{1/2}(\partial K,\mu)$, see e.g. \cite[Th.6.2]{SobBook}. As a consequence, we may find a constant $C_{2} = C_{2}(K,\mu)$ such that
\begin{equation}\label{eq:frac_boundary_embedding}
     \|\rho\|_{H^{1/2}(\partial K,\mu)}
     \le C_2\,\|\rho\|_{H^{1}(\partial K,\mu)}
\end{equation}
for all $\rho \in C^1(\partial K)$. For the rest of the section, $c,\tilde{c},C,\text{etc}$ will denote constants which depend only on the measure $\mu$ and the convex set $K$ and whose value may change at different steps. 

Since we are interested in applying \eqref{eq:multiplicative_form} to functions in $H^{1}(\partial K,\mu)$ and $H^{1}(K,\mu)$, let us justify that this passage is indeed possible. Indeed, usual approximation arguments via smooth functions will suffice. 
\begin{lemma}\label{lem:continuity_both_forms}
    Suppose that $K$ and $\mu$ satisfy the assumptions of Theorem~\ref{thm:multiplicative_form}. Then there exist constants $c,\tilde{c}$ which depend only on the measure and on $K$ such that
    \begin{equation}\label{eq:continuity_poincare}
    \langle \rho ,\rho \rangle_{\rm P} \leq c\|\rho\|^2_{H^{1}(\partial K,\mu)} \,\, \text{for all} \,\, \rho \in H^{1}(\partial K,\mu),
    \end{equation}
    and 
    \begin{equation}\label{eq:continuity_BL}
        \langle \psi, \psi \rangle_{\rm BL} 
        \leq  \tilde{c} \|\psi\|^2_{H^{1}(K,\mu)} \,\, \text{for all} \,\,\psi \in H^1(K,\mu).
    \end{equation}
Moreover, the interaction term $\langle \cdot, \cdot \rangle_{\rm I}$ is continuous in $H^{1}(\partial K,\mu)\times H^{1}(K,\mu)$.    
\end{lemma}
\begin{proof} Let us begin by noting that since $K$ is compact there exists an $R>0$ such that $K \subseteq B_{R}(0)$. Since the density of $\mu$ is strictly log-concave and $C^2$ smooth, we may find strictly positive constants $k_{1}$ and $k_{2}$ such that 
\begin{equation}\label{eq:restricted_Hessian_pinching} \notag
k_{1}\mathrm{Id} \leq \nabla^2u(x) \leq k_{2}\mathrm{Id}, \,\, \text{for all} \,\, x \in B_{R}(0).
\end{equation}
The above implies, in particular, that $\|\nabla u\|$ restricted to $B_{R}(0)$ is Lipschitz with constant bounded by $k_{2}$, thus in particular
\begin{equation}\label{eq:Lipschitz_restriction}
\|\nabla u(x)\| \leq k_{2}\|x\|+\|\nabla u(0)\| \,\, \text{for all} \,\, x \in \partial K.
\end{equation}
In addition, since $K$ is strictly convex and compact, the minimal and maximal curvatures of $\partial K$
\begin{equation}\label{eq:min_max_curvature} \notag
m_{1} = \min_{x \in \partial K, \,\,i=1,\dots,n-1} \kappa_{i}(x) \,\, \text{and} \,\, m_{2} = \max_{x \in \partial K, \,\,i=1,\dots,n-1} \kappa_{i}(x)
\end{equation}
are strictly positive and finite. In particular, the second fundamental form of the boundary satisfies $ m_{1} \mathrm{Id} \leq {\rm{II}} \leq m_{2}\mathrm{Id}$ and thus ${\rm{II}}^{-1} \leq m^{-1}_{1}\mathrm{Id} $ as well as $\mathrm{Tr}({\rm{II}}) \leq (n-1)m_{2}$. Furthermore, for the generalized mean curvature notice that \eqref{eq:Lipschitz_restriction} implies, for every $x \in \partial K$,
\begin{align} \notag 
    |{\rm H}_{\mu}| &= |\tra({\rm{II}}) - \langle \nabla u, \nu_{K}\rangle | \leq \tra({\rm{II}})+ \|\nabla u \| \\
    &\leq (n-1)m_{2}+c_{2}\|x\|+ \|\nabla u(0)\| \leq (n-1)m_{2}+c_{2}R +c \leq C  \label{eq:Generalised_Mean_Estimate}
\end{align}
With these preliminary estimates in place, notice that
\begin{align*}
    \langle \rho , \rho \rangle_{\rm P} &\leq \int_{\partial K} m^{-1}_{1}\|\nabla_{\partial K}\rho\|^2\di \mu-\int_{\partial K}{\rm H}_{\mu}\rho^2\di \mu+\frac{1}{\mu(K)}\left(\int_{\partial K} \rho \di \mu\right)^2 \\
    &\leq m^{-1}_{1}\int_{\partial K}\|\nabla_{\partial K}\rho\|^2\di \mu + \int_{\partial K} |{\rm H}_{\mu}|\rho^2\di \mu +\frac{\mu(\partial K)}{\mu(K)}\int_{\partial K} \rho^2 \di \mu \\
    &\leq c \|\rho\|^2_{H^{1}(\partial K,\mu)},
\end{align*}
where in the second line we applied Cauchy--Schwarz as well as the estimate in \eqref{eq:Generalised_Mean_Estimate}. On the other hand notice that since $\nabla^2u \geq k_{1}\mathrm{Id}$ we have that, equivalently, $(\nabla^2u)^{-1} \leq k^{-1}_{1}\mathrm{Id}$. Thus, the proof is concluded by observing that
\begin{align}\label{eq:Prop 3.1.1}
        \notag 
        \langle \psi,\psi \rangle_{\rm BL}
        &\leq k^{-1}_1\!\int_{K}\|\nabla \psi \|^2 \, \di\mu
            - \int_{K}\psi^2\, \di\mu
            + \frac{1}{\mu(K)}\left(\int_{K}\psi \, \di\mu\right)^2  
        \\
        \notag
        &\leq k^{-1}_1\int_{K} \|\nabla\psi\|^2 \di \mu \leq k^{-1}_1\!\left[\int_{K} \big(\|\nabla\psi\|^2 + \psi^2\big)\, \di\mu\right] 
        \\
        &= k^{-1}_1 \|\psi \|^{2}_{H^{1}(K,\mu)}.
    \end{align}
For the final claim of the lemma, recall that for every $\psi \in H^{1}(K,\mu),\rho \in H^{1}(\partial K,\mu) $, we have
\[\langle \rho, \psi\rangle_{\rm I} = \int_{\partial K} \rho\psi\di\mu -\frac{1}{\mu(K)}\int_{\partial K}\rho\di\mu\int_{K}\psi\di\mu\]
it just remains to observe that by the continuity of the trace operation regarded as a function $Tr:H^{1}(K,\mu)\rightarrow L^2(\partial K, \mu)$ (where $L^2(\partial K, \mu)$ denotes the space of square integrable functions on $\partial K$ with respect to $\mu$), every boundary term in the expression is continuous in $L^{2}(\partial K, \mu)$, thus in particular in $H^1(\partial K,\mu)$, since we have no derivative terms. Finally since the term $\int_{K}\psi \di \mu$ is also continuous in $L^2(K, \mu)$, as $K$ has finite measure, $\langle \cdot,\cdot\rangle_{\rm I}$ is continuous in the product space. This concludes the proof of the lemma.
\end{proof}

We now obtain a technical estimate, resembling an interpolation inequality.
\begin{proposition}\label{prop:Boundary Sobolev form}
    Assume that $K$ and $\mu$ satisfy the assumptions of Theorem~\ref{thm:multiplicative_form}. Then there exists a constant $c>0$, depending only on $K$ and $\mu$, such that
    \[
    \|\rho\|^2_{L^2(\partial K, \mu)}
    \;\le\;
    c\, \langle \rho,\rho \rangle_P^{1/2}\,
       \|\rho\|_{H^{1}(\partial K,\mu)}
    \qquad\text{for all } \rho \in H^{1}(\partial K,\mu).
    \]
\end{proposition}
\begin{proof}
We begin by noticing that combining \eqref{eq:continuity_BL} with \eqref{eq:multiplicative_form}, for every $\psi \in H^{1}(K,\mu)$ such that $\int_{K}\psi \, \di \mu = 0$ and every $\rho \in H^{1}(\partial K,\mu)$ we have
    \begin{equation}\label{eq:Prop 3.1.2}
        \int_{\partial K} \rho\,\psi \, d\mu 
        \;\le\;
        \langle \rho,\rho \rangle_{\rm P}^{1/2}\,
        \langle \psi,\psi \rangle_{\rm BL}^{1/2}
        \;\le\;
        c\,\langle \rho,\rho \rangle_{\rm P}^{1/2}
        \|\psi \|_{H^{1}(K,\mu)} .
    \end{equation}
Fix now $\rho \in H^{1}(\partial K,\mu)$. If $\|\rho\|_{L^2(\partial K, \mu)}=0$ then the statement is trivial. Suppose instead that $\|\rho\|_{L^2(\partial K, \mu)}\neq0$. Then, by \eqref{eq:fractional_extension}, we can find $\varphi \in C^{1}(K)$ such that $Tr(\varphi) = \rho$ and
    \begin{equation}\label{eq:Prop 3.1.3}
        \|\varphi\|_{H^{1}(K,\mu)}
        \leq C_1 \|\rho\|_{H^{1/2}(\partial K,\mu)} .
    \end{equation}
Consider $g \in C^{1}(K)$ such that $g|_{\partial K} =0$ and $\int_{K} g\, \di\mu \neq 0$, and define
    \[\tilde{\varphi}(x) =  \varphi(x)- \frac{\int_K \varphi\, \di\mu}{\int_{K} g\, \di\mu}\, g(x). \]
    Notice that $\tilde{\varphi}$ still satisfies $Tr(\tilde{\varphi})= \rho$ and in addition $\int_{K}\tilde{\varphi}\, \di\mu =0$.  
    Choosing the constant $\tilde{c} =  \|g\|_{H^{1}(K,\mu)}/|\int_{K}g\, \di\mu|$, we obtain the estimate
    \begin{align}\label{eq:Prop 3.1.4}
        \notag
        \|\tilde{\varphi}\|_{H^{1}(K,\mu)}
        &\leq \|\varphi\|_{H^{1}(K,\mu)}
         + \tilde{c}\,\left|\!\int_{K}\varphi\, \di \mu\right|        \notag
        \leq \|\varphi\|_{H^{1}(K,\mu)}
          + \tilde{c}\sqrt{\mu(K)}\,\|\varphi\|_{L^2(K, \mu)}
        \\
        &\leq \left(1+C\sqrt{\mu(K)}\right)\,\|\varphi\|_{H^{1}(K,\mu)}
        \leq \tilde{C}\|\rho\|_{H^{1/2}(\partial K,\mu)},
    \end{align}
    where the last inequality follows by \eqref{eq:Prop 3.1.3}. Observe also that since $g$ was arbitrary and it does not depend on $\rho$, all the constants depend only on $K$ and $\mu$. We may now choose $\psi = \tilde{\varphi}$ in \eqref{eq:Prop 3.1.2} and, by \eqref{eq:Prop 3.1.4},
    \[
    \|\rho\|^2_{L^2(\partial K, \mu)}
    \;\leq\;
    c\,
    \langle \rho,\rho \rangle_{\rm P}^{1/2}\,
    \|\rho\|_{H^{1/2}(\partial K,\mu)}. 
    \] 
    Finally, thanks to \eqref{eq:frac_boundary_embedding}, we have
    \[
    \|\rho\|^2_{L^2(\partial K, \mu)}
    \;\leq\;
    \tilde{C}\,
    \langle \rho,\rho \rangle_{\rm P}^{1/2}\,
    \|\rho\|_{H^{1}(\partial K,\mu)},
    \]
    completing the proof.
\end{proof}

We can now prove the main result of this section (Theorem \ref{thm:stability_intro} in the introduction). 
\begin{theorem}\label{thm:Colesanti_Stability}
Assume that $K$ and $\mu$ satisfy the assumptions of Theorem~\ref{thm:multiplicative_form}. If $\rho \in H^1(\partial K,\mu)$ is such that 
\begin{equation}\label{stability_statement}
\int_{\partial K} \left\langle {\rm{II}}^{-1}\nabla_{\partial K}\rho, \nabla_{\partial K} \rho \right\rangle \di \mu \leq \int_{\partial K}\mathrm{H}_{\mu}\rho^2\di \mu  -\frac{1}{\mu(K)}\left( \int_{\partial K}\rho \di \mu  \right)^2 + \varepsilon 
\end{equation}
for some $\varepsilon \in (0,1)$, then
\[\|\rho\|_{H^{1}(\partial K,\mu)} \leq C\varepsilon^{1/2}\]
for a constant $C$ depending only on $K$ and $\mu$.
\end{theorem}
\begin{proof} We claim that, under the assumptions of the theorem,
\begin{equation}\label{eq:Sobolev vs Poincare}
\|\rho\|_{H^{1}(\partial K,\mu)} \leq c\|\rho\|_{L^2(\partial K, \mu)}+\tilde{c}\varepsilon^{1/2}.
\end{equation} 
Indeed, denoting by $m_{2} >0$ the maximal curvature of $\partial K$,
 \begin{equation}\label{eq:stability_1}
        m_{2}^{-1}\int_{\partial K} \|\nabla_{\partial K} \rho \|^2 \di \mu 
        \leq \int_{\partial K} \left\langle {\rm{II}}^{-1} \nabla_{\partial K } \rho, \nabla_{\partial K } \rho \right\rangle \di \mu.
    \end{equation}
On the other hand, by \eqref{eq:Generalised_Mean_Estimate},
\begin{equation}\label{eq:stability_2}
    \int_{\partial K} {\rm H}_\mu \rho^2 \di \mu \leq \int_{\partial K} |{\rm H}_\mu| \rho^2 \di \mu \leq C\|\rho\|^2_{L^2(\partial K, \mu)} \, .
\end{equation}
Dropping the (non-positive) squared mean in \eqref{stability_statement} and combining \eqref{eq:stability_1} with \eqref{eq:stability_2}, we infer
\[\int_{\partial K} \|\nabla_{\partial K} \rho \|^2 \di \mu 
        \leq c\|\rho\|^2_{L^2(\partial K, \mu)}+ \tilde{C}\varepsilon. \]
Finally, we have that
\[\|\rho\|_{H^{1}(\partial K,\mu)} = \left[ \int_{\partial K} \|\nabla_{\partial K} \rho \|^2+\rho^2\di \mu \right]^{1/2} \leq \left[c\|\rho\|^2_{L^2(\partial K, \mu)} +\tilde{c}\varepsilon \right]^{1/2} \leq 
C\|\rho\|_{L^2(\partial K, \mu)} + \tilde{C}\varepsilon^{1/2},\]
proving our claim. Combining the above with Proposition~\ref{prop:Boundary Sobolev form} yields
\[\|\rho\|^{2}_{L^2(\partial K, \mu)} \leq c_{1} \cdot \varepsilon^{1/2}\|\rho\|_{L^2(\partial K, \mu)}+c_{2} \varepsilon \]
for some constants $c_{1},c_{2},$ which depend only on $K,\mu$. This implies that $\|\rho\|_{L^2(\partial K, \mu)}$ lies between the roots of $x^2-c_{1}\varepsilon^{1/2}x-c_{2}\varepsilon=0$. Therefore, $\|\rho\|_{L^2(\partial K, \mu)}$ is less than the maximal root. I.e.,
\[\|\rho\|_{L^2(\partial K, \mu)} \leq \frac{\varepsilon^{1/2}}{2}\left(c_{1}+\sqrt{c^2_{1}+4c_{2}}\right).  \]
The proof is concluded with another application of \eqref{eq:Sobolev vs Poincare}.
\end{proof}

As an immediate corollary, we deduce the equality cases of the weighted Poincar\'{e} inequalities in the boundary of convex sets.
\begin{corollary}\label{cor:eq_case}
Suppose that the function $\rho \in C^{1}(\partial K)$ satisfies
\[\int_{\partial K} \left\langle {\rm{II}}^{-1}\nabla_{\partial K}\rho, \nabla_{\partial K} \rho \right\rangle \di \mu =\int_{\partial K}{\rm H}_\mu\rho^2\di \mu  -\frac{1}{\mu(K)}\left( \int_{\partial K}\rho \di \mu  \right)^2 , \]
where $\mu$ and $K$ satisfy the assumptions of Theorem~\ref{thm:multiplicative_form}. Then $\rho(x)=0$, for all $x \in \partial K$.
\end{corollary}\noindent
The conclusion of Corollary \ref{cor:eq_case} is in stark contrast to the equality cases of the unweighted version established by Colesanti in \cite{Colesanti_poincare}. The equality cases (or absence thereof) of the weighted Poincar\'{e} inequality are evidence of the fact that this inequality can be improved in various ways, as already discussed in \cite{Kol_Mil_P}. The improvement of these infinitesimal forms is extremely natural (see Proposition \ref{Prop:Bilinear_vs_gradient norm} later), and lies at the core of the recent treatment of concavity properties of log-concave measures, which will be the topic of the next section. 

\section{Concavity powers of strictly log-concave measures}

The purpose of this section is to present applications of the inequality in Theorem \ref{thm:multiplicative_form} to concavity properties of functionals restricted to sets in the class $\mathcal{K}^n$. In the present section it will be convenient to normalize $\mu$ as a probability measure supported on $\mathbb{R}^n$. We adopt some terminology and observations from various authors; see, for example, \cite[Section 3]{GalynaStability} and the references therein. We begin the exposition in a general, not necessarily symmetric, setting.

\medskip

We are interested in the largest value $p(\mu)$ such that 
\begin{equation}\label{eq:pmu}
    \mu((1-t)K+tL)^{p(\mu)} \geq (1-t)\mu(K)^{p(\mu)}+t\mu(L)^{p(\mu)}
\end{equation}
for every $K,L \in \K^n$ and $t \in [0,1]$. Notice that, as $\mu$ is log-concave, \eqref{eq:pmu} is, by definition, true for the power $p=0$. Moreover, H\"older's inequality implies that if \eqref{eq:pmu} is true for some $p(\mu)>0$, then it remains true for every $p \in [0,p(\mu))$. As mentioned in the introduction, it is conjectured that the concavity exponent in \eqref{eq:pmu} is $1/n$ if one restricts the inequality to an even log-concave probability measure and origin-symmetric convex bodies. That is, \eqref{eq:dim_BM} holds. This conjecture was verified in the contexts presented in \cite{dim_gauss, rot_B, cordero2025concavity}, but a full solution is yet to be found.  

We briefly explain how estimating  $p(\mu)$ can be approached via a local to global principle, following \cite{Kol_Liv_Gardner-Zvavitch,GalynaStability}. Fix a convex set $K$ and define the concavity power of $K$ with respect to $\mu$ by
\begin{equation}\label{eq:local_concavity_power}
p(\mu,K)= \limsup_{\varepsilon \rightarrow 0^{+}} \left\{p>0 : \forall L \in \mathcal{K}^{n},\ \mu((1-\varepsilon)K+\varepsilon L)^{p} \geq (1-\varepsilon)\mu(K)^{p}+\varepsilon\mu(L)^{p}\right\}. 
\end{equation}
Then it follows that (see, e.g., \cite{GalynaStability})
\[p(\mu) = \inf_{K \in \K^n} p(\mu,K).\]
The symmetric analogues are defined similarly, with the obvious modifications. Therefore, the general strategy is to prove that for every convex set $K$ one has $p(\mu,K) \geq c(\mu)$ for some $c(\mu)$ independent of $K$, and then to deduce the global concavity estimate \eqref{eq:pmu}.

\medskip

We now explain how the quantity in \eqref{eq:local_concavity_power} is connected to spectral inequalities. Fix a smooth and strictly convex $K\in\mathcal{K}^n$, and let $L\in\mathcal{K}^n$ be smooth and strictly convex as well. Define the function $f : \sph^{n-1} \to \mathbb{R}$ by
\begin{equation}\label{eq:addmisible_pertubations}
f = h_{L}-h_{K},
\end{equation}
where $\nu_K$ denotes the outer unit normal to $\partial K$. Then
\[(1-t)K + t L = [h_{K}+tf].\]
With this notation, \eqref{eq:local_concavity_power} can be rewritten as
\[p(\mu,K)= \sup_{p \in (0,\infty)} \left\{ p : \left.\frac{d^{2}}{dt^{2}}\mu(K(f,t))^{p}\right|_{t=0} \leq 0\quad \forall f \in C^{1}(\sph^{n-1}), \text{ where } K(f,t)=[h_K+tf] \right\},\] where we have that linear combinations of differences of support functions in $C^2(\sph^{n-1})$ are dense in $C^1(\sph^{n-1})$. With the notation $\rho=f(\nu_K)$ (which is in $C
^1(\partial K)$ by the regularity of $K$), if we now write \[J(\rho) = \frac{\mu(K)\langle \rho,\rho\rangle_{\rm P}}{\left(\int_{\partial K}\rho\,\di\mu\right)^2},\] where we recall that the bilinear form appearing on the right-hand side is the one defined in \eqref{eq:P}, explicit computations (compare again \cite{GalynaStability}) show that 
\begin{equation}\label{eq:def_concavity_power}
p(\mu,K)= \inf_{\rho \in C^{1}(\partial K)} J(\rho).
\end{equation}

In the case where the convex sets $K$ and $L$ are origin symmetric, the admissible perturbations $f$ in \eqref{eq:addmisible_pertubations} must be even. Consequently, the infimum in \eqref{eq:def_concavity_power} is taken over all even functions in $C^{1}(\partial K)$. Using the $L_2$ method, the functional $J(\rho)$ can be estimated from below using only integrals over $K$; see, for example, Proposition~3.16 in \cite{GalynaStability}. This has been the prevailing approach in all prior results in the literature.

As mentioned in the introduction, we shall propose a different approach, working directly with the functional $J(\rho)$. In particular, we shall prove that the infimum in \eqref{eq:def_concavity_power} is in fact attained, and compute the associated Euler--Lagrange equation. By a standard density argument,
\begin{equation}\label{eq:concavity_power_Sobolev}
p(\mu,K)= \inf_{\rho \in H^{1}(\partial K,\mu)} J(\rho),
\end{equation}
so that we may work in the Sobolev space $H^{1}(\partial K,\mu)$, which is better suited to our purposes since it is complete. The idea of deducing concavity principles via an associated partial differential equation is not new; see, for example, \cite{Colesanti_eigenvalue_concavity,Borell_Torsional_rigidity} for analogous approaches concerning the first Dirichlet eigenvalue of the Laplacian and torsional rigidity, as well as \cite[Section 4]{Salani_elliptic_PDE} for a broader discussion. However, to the best of our knowledge, this approach is new in the context of concavity powers of convex sets, partly because the corresponding Euler--Lagrange equation had not previously been available. In this sense, our method yields an explicit characterization of the concavity power of a given set $K$ through a precise perturbation, incorporating the structure of the underlying measure $\mu$.

\medskip

\paragraph{Existence of a minimizer and the Lax--Milgram theorem.} Now that we have concluded the necessary preparations, we shall prove the existence of a minimizer in \eqref{eq:concavity_power_Sobolev}. At the core of our approach lies the Lax--Milgram theorem. In order to apply it, we need to show that the symmetric bilinear form $\langle \cdot, \cdot \rangle_{\rm P}$ is continuous and coercive with respect to the norm $\|\cdot\|_{H^{1}(\partial K,\mu)}$. Since continuity was established in Lemma~\ref{lem:continuity_both_forms}, it remains to prove coercivity. A fundamental step is the following strengthened form of the Poincar\'e inequality \eqref{eq:Poincare}.
\begin{proposition}\label{Prop:Bilinear_vs_gradient norm}
Let $K,\mu$ satisfy the assumptions of Theorem~\ref{thm:multiplicative_form}. Then there exists a constant $\lambda_{1}=\lambda_{1}(K,\mu)>1$ such that for every $\rho \in H^{1}(\partial K,\mu)$,
\begin{equation}\label{eq:Boundary_Improved_Poincare}
   \int_{\partial K}{\rm H}_\mu\rho^2\,\di\mu -\frac{1}{\mu(K)}\left( \int_{\partial K}\rho\di\mu \right)^2 \leq  \frac{1}{\lambda_1}\int_{\partial K} \left\langle {\rm II}^{-1}\nabla_{\partial K}\rho, \nabla_{\partial K} \rho \right\rangle \di\mu.
\end{equation}
\end{proposition}
\begin{proof}
Observe first that the inequality holds trivially for any $\lambda_1$ whenever the left-hand side is non-positive, since by the strict convexity of $K$ one has that ${\rm II}^{-1}$ is positive definite. We may therefore restrict our attention to functions $\rho\in H^{1}(\partial K,\mu)$ for which the left-hand side of \eqref{eq:Boundary_Improved_Poincare} is positive. Denote this class by $\mathcal{P}$ and define
\[\lambda_{1}= \inf_{\rho \in \mathcal{P}} \frac{\int_{\partial K} \left\langle {\rm II}^{-1}\nabla_{\partial K}\rho, \nabla_{\partial K} \rho \right\rangle \di\mu}{\int_{\partial K}{\rm H}_\mu\rho^2\,\di\mu-\frac{1}{\mu(K)}\left( \int_{\partial K}\rho\,\di\mu \right)^2 }.\]
We claim that $\lambda_{1}>1$. Since $\langle \cdot,\cdot\rangle_{\rm P}\geq 0$, it follows immediately that $\lambda_{1}\geq 1$. Assume, by contradiction, that $\lambda_{1}=1$. Then for every $\varepsilon>0$ there would exist a function $\rho_{\varepsilon}\in\mathcal{P}$ such that
\[ \int_{\partial K} \left\langle {\rm II}^{-1}\nabla_{\partial K}\rho_{\varepsilon}, \nabla_{\partial K} \rho_{\varepsilon}\right\rangle \,\di\mu\leq (1+\varepsilon)\left[\int_{\partial K}{\rm H}_\mu\rho_{\varepsilon}^2\,\di\mu-\frac{1}{\mu(K)}\left( \int_{\partial K}\rho_{\varepsilon}\,\di\mu \right)^2\right]. \]
Thus,
\begin{align*}
    \notag \langle \rho_{\varepsilon},\rho_{\varepsilon} \rangle_{\rm P}
    &\leq \varepsilon\left[\int_{\partial K}{\rm H}_\mu\rho_{\varepsilon}^2\,\di\mu-\frac{1}{\mu(K)}\left( \int_{\partial K}\rho_{\varepsilon}\,\di\mu \right)^2
    \right] \\
    &\leq \varepsilon \int_{\partial K}|{\rm H}_\mu|\rho_{\varepsilon}^2\,\di\mu \leq c\,\varepsilon\,\|\rho_{\varepsilon}\|^2_{L^{2}(\partial K,\mu)},
\end{align*}
where in the second inequality we dropped the non-positive term and used \eqref{eq:Generalised_Mean_Estimate}. Rearranging the terms in the previous inequality we infer
\begin{equation}\label{eq:Poincare_estimate}
     \langle \rho_{\varepsilon},\rho_{\varepsilon} \rangle_{\rm P} \leq c\,\varepsilon\,\|\rho_{\varepsilon}\|^2_{L^{2}(\partial K,\mu)}
\end{equation}
for every $\varepsilon >0$. Since \eqref{eq:Poincare_estimate} is homogeneous, we may assume, without loss of generality, that $\|\rho_{\varepsilon}\|_{L^{2}(\partial K,\mu)}=1$, which yields
\[\langle \rho_{\varepsilon},\rho_{\varepsilon} \rangle_{\rm P} \leq c\,\varepsilon.\]
By Theorem~\ref{thm:Colesanti_Stability}, this implies
\[1=\|\rho_{\varepsilon}\|_{L^{2}(\partial K,\mu)}\leq \|\rho_{\varepsilon}\|_{H^{1}(\partial K,\mu)}\leq C\sqrt{c\varepsilon}= \tilde C\,\varepsilon^{1/2},\]
which, since $\varepsilon>0$ is arbitrary, leads to a contradiction. Therefore, $\lambda_{1}>1$, concluding the proof.
\end{proof}\noindent
We can now show that $\langle \cdot, \cdot\rangle_{\rm P}$ is coercive.
\begin{lemma}\label{lem:poincare_coercivity}
Let $K,\mu$ satisfy the assumptions of Theorem~\ref{thm:multiplicative_form}. Then there exists a constant $C=C(K,\mu)>0$ such that for every $\rho\in H^{1}(\partial K,\mu)$,
\begin{equation}\label{eq:control_of_H1_norm}
    \langle \rho,\rho \rangle_{\rm P}\geq C \|\rho\|^2_{H^{1}(\partial K,\mu)}.
\end{equation}
\end{lemma}
\begin{proof}
We claim that there exists $c>0$ such that
\[\langle \rho,\rho \rangle_{\rm P}\geq c\,\max\left\{\|\rho\|^2_{L^2(\partial K, \mu)},\|\nabla_{\partial K}\rho\|^2_{L^2(\partial K, \mu)}\right\}\quad\text{for all} \,\,\,\rho\in H^{1}(\partial K,\mu).\]
We fix $\rho\in H^{1}(\partial K,\mu)$ and consider two cases.
\smallskip

\noindent
\emph{Case 1.}
Suppose that
\[\max\left\{\|\rho\|^2_{L^2(\partial K, \mu)},\|\nabla_{\partial K}\rho\|^2_{L^2(\partial K, \mu)}\right\}=\|\rho\|^2_{L^2(\partial K, \mu)}.
\]
Using Proposition~\ref{prop:Boundary Sobolev form}, we obtain
\[\|\rho\|^4_{L^2(\partial K, \mu)}\leq c_1 \langle \rho,\rho \rangle_{\rm P}\,\|\rho\|^2_{H^{1}(\partial K,\mu)}\leq 2c_1 \langle \rho,\rho \rangle_{\rm P}\,\|\rho\|^2_{L^2(\partial K, \mu)}.\]
Hence,
\begin{equation}\label{eq:coercivity_case_1}
\|\rho\|^2_{L^2(\partial K, \mu)}
\leq 
2c_1\,\langle \rho,\rho \rangle_{\rm P}.
\end{equation}

\smallskip
\noindent
\emph{Case 2.}
Suppose instead that
\[\max\left\{\|\rho\|^2_{L^2(\partial K, \mu)},\|\nabla_{\partial K}\rho\|^2_{L^2(\partial K, \mu)}\right\}=\|\nabla_{\partial K}\rho\|^2_{L^2(\partial K, \mu)}.\]
Since $\lambda_{1}>1$, by Proposition~\ref{Prop:Bilinear_vs_gradient norm} we can write
\begin{align}
\langle \rho,\rho \rangle_{\rm P}
&=\left(1-\frac{1}{\lambda_1}\right)\int_{\partial K} \left\langle {\rm II}^{-1}\nabla_{\partial K}\rho, \nabla_{\partial K} \rho \right\rangle \di \mu+ \frac{1}{\lambda_1}\int_{\partial K} \left\langle {\rm II}^{-1}\nabla_{\partial K}\rho, \nabla_{\partial K} \rho \right\rangle \di\mu  \notag \\
\notag &- \int_{\partial K}{\rm H}_\mu\rho^2\,\di\mu +\frac{1}{\mu(K)}\left( \int_{\partial K}\rho\di\mu \right)^2 \\
&\geq \left(1-\frac{1}{\lambda_1}\right)\int_{\partial K} \left\langle {\rm II}^{-1}\nabla_{\partial K}\rho,\nabla_{\partial K} \rho \right\rangle \di \mu
\geq c_2 \|\nabla_{\partial K}\rho\|^2_{L^2(\partial K, \mu)},
\end{align}
where \eqref{eq:min_max_curvature} was also used in the last step. Thus,
\begin{equation}\label{eq:coercivity_case_2}
\langle \rho,\rho \rangle_{\rm P}\geq c_2 \max\left\{\|\rho\|^2_{L^2(\partial K, \mu)},\|\nabla_{\partial K}\rho\|^2_{L^2(\partial K, \mu)}\right\}.
\end{equation}
Combining both cases yields the claim. The conclusion \eqref{eq:control_of_H1_norm} follows from the trivial inequality
\[\|\rho\|^2_{H^{1}(\partial K,\mu)}\leq 2\max\left\{\|\rho\|^2_{L^2(\partial K, \mu)},\|\nabla_{\partial K}\rho\|^2_{L^2(\partial K, \mu)}\right\}.\]
\end{proof}

\medskip

The main result of this section (Theorem \ref{thm:minimal_power} in the introduction) reads as follows. We recall that for a vector field $X \in C^1(\partial K,\R^n)$, $\nabla_{\partial K} \cdot X$ is the tangential divergence of $X$.
\begin{theorem}\label{thm:Section_4_main1}
Let $K$ and $\mu$ satisfy the assumptions of Theorem~\ref{thm:multiplicative_form}. Then there exists a unique function $\overline{\rho} \in H^{1}(\partial K,\mu)$ which is the weak solution of
\begin{equation}\label{eq:PDE_divergence_form}
-\nabla_{\partial K} \cdot ({\rm II}^{-1}\nabla_{\partial K}\rho)+\langle \nabla_{\partial K}u,{\rm II}^{-1}\nabla_{\partial K}\rho \rangle -{\rm H}_\mu\,\rho+\frac{1}{\mu(K)}\int_{\partial K}\rho\,\di\mu=1\qquad\text{on }\partial K.
\end{equation}
Furthermore, the concavity power of $K$ with respect to $\mu$ is given by
\begin{equation}\label{eq:new_concavity expression}
p(\mu,K) = \frac{\mu(K)}{\int_{\partial K} \overline{\rho}\,\di\mu}.
\end{equation}
\end{theorem}

\begin{proof}
Consider the bounded linear functional $l : H^{1}(\partial K,\mu) \to \mathbb{R}$ defined by
\[l(\rho) = \int_{\partial K} \rho\,\di\mu.\]
Since the bilinear form $\langle \cdot,\cdot \rangle_{\rm P}$ is continuous (Lemma~\ref{lem:continuity_both_forms}) and coercive (Lemma~\ref{lem:poincare_coercivity}), the Lax--Milgram theorem guarantees the existence of a unique function $\overline{\rho} \in H^{1}(\partial K,\mu)$ such that
\begin{equation}\label{eq:PDE_weak}
\langle \overline{\rho}, \rho \rangle_{\rm P}= \int_{\partial K}\rho\,\di\mu,\qquad \text{for all} \,\,\, \rho \in H^{1}(\partial K,\mu).
\end{equation}

For every $\varphi \in C^{1}(\partial K)$ and every sufficiently smooth tangent vector field $X$, an integration by parts yields
\[\int_{\partial K} \langle \nabla_{\partial K}\varphi , X \rangle \,\di\mu=-\int_{\partial K} \varphi\left(\nabla_{\partial K} \cdot X-\langle \nabla_{\partial K}u,X \rangle\right)\di\mu.\]
Thus, the weak formulation \eqref{eq:PDE_weak} is equivalent to \eqref{eq:PDE_divergence_form}, which proves the existence and uniqueness of a weak solution $\overline{\rho}$. For the second part, applying the Cauchy--Schwarz inequality with respect to the bilinear form $\langle \cdot,\cdot \rangle_{\rm P}$ yields, for all $\rho \in H^{1}(\partial K,\mu)$,
\[l(\rho)^2=\langle \overline{\rho}, \rho \rangle_{\rm P}^2\leq\langle \overline{\rho}, \overline{\rho} \rangle_{\rm P}\langle \rho, \rho \rangle_{\rm P},\]
and a rearrangement of the terms gives
\begin{equation}\label{eq:Lax_Cauchy}
\frac{1}{\langle \overline{\rho}, \overline{\rho} \rangle_{\rm P}}
\leq
\frac{\langle \rho, \rho \rangle_{\rm P}}{\left(\int_{\partial K}\rho\,\di\mu\right)^2},
\qquad
\text{for all} \,\,\, \rho \in H^{1}(\partial K,\mu).
\end{equation}
Since equality holds for $\rho=\overline{\rho}$, we conclude that
\[\inf_{\rho \in C^{1}(\partial K)}\frac{\langle \rho,\rho \rangle_{\rm P}}{\left(\int_{\partial K}\rho\,\di\mu\right)^2}=\inf_{\rho \in H^{1}(\partial K,\mu)}\frac{\langle \rho,\rho \rangle_{\rm P}}{\left(\int_{\partial K}\rho\,\di\mu\right)^2}=\frac{1}{\langle \overline{\rho}, \overline{\rho} \rangle_{\rm P}}.\]
Finally, choosing $\rho=\overline{\rho}$ in \eqref{eq:PDE_weak} yields
\[\langle \overline{\rho}, \overline{\rho} \rangle_{\rm P}=\int_{\partial K}\overline{\rho}\,\di\mu,\]
which, together with \eqref{eq:concavity_power_Sobolev}, completes the proof.
\end{proof}

\medskip

If we assume symmetry with respect to the origin for $K$ and that $\mu$ is even, we obtain the following corollary of Theorem \ref{thm:Section_4_main1}. These assumptions are precisely those of the dimensional Brunn--Minkowski conjecture \eqref{eq:dim_BM}, which is known to be false for non-symmetric sets and non-even measures. Compare, e.g., \cite{nayar2013note}.
\begin{corollary}\label{cor:even}
Let $K,\mu$ satisfy the assumptions of Theorem~\ref{thm:multiplicative_form}. Suppose, in addition, that $K$ is origin-symmetric and that the measure $\mu$ is even. Then the function $\overline{\rho}$ solving \eqref{eq:PDE_weak} is even.
\end{corollary}
\begin{proof}
Since $\mu$ is even and $K$ is symmetric, one has
\[\langle \rho, \sigma \rangle_{\rm P} = 0\]
whenever $\rho$ is even and $\sigma$ is odd. Decompose
\[\overline{\rho} = \overline{\rho}_{\mathrm{even}} + \overline{\rho}_{\mathrm{odd}}.\]
Choosing $\rho=\overline{\rho}_{\mathrm{odd}}$ in \eqref{eq:PDE_weak} gives
\[\langle \overline{\rho}_{\mathrm{odd}},\overline{\rho}_{\mathrm{odd}} \rangle_{\rm P}+\langle\overline{\rho}_{\mathrm{even}},\overline{\rho}_{\mathrm{odd}} \rangle_{\rm P}=\int_{\partial K} \overline{\rho}_{\mathrm{odd}}\,\di\mu=0,\]
where the last equality follows from symmetry. Since the mixed term vanishes, we obtain
\[\langle \overline{\rho}_{\mathrm{odd}},\overline{\rho}_{\mathrm{odd}} \rangle_{\rm P}=0.\]
By the equality cases characterized in Theorem~\ref{thm:Colesanti_Stability}, this implies $\overline{\rho}_{\mathrm{odd}}=0$. Thus, $\overline{\rho}$ is even.
\end{proof}\noindent
Observe that, under the assumptions of Corollary \ref{cor:even}, we can write \eqref{eq:concavity_power_Sobolev} as \[ p(\mu,K)=\inf \{ J(\rho): \rho \in H^{1}(\partial K,\mu) \text{ and }\rho \text{ is even} \}.\] In the notation of \cite{GalynaStability}, $p_s(\mu,K)=p(\mu,K)$ or, in other words, restricting the space of admissible perturbations to even ones in \eqref{eq:concavity_power_Sobolev} is a consequence of the underlying symmetry. 

\medskip

\paragraph{A reformulation of the dimensional Brunn--Minkowski conjecture.} We now rewrite \eqref{eq:PDE_divergence_form}. We use the notation \[ \nabla_{\partial K,\mu}\cdot X = \nabla_{\partial K} \cdot X-\langle \nabla_{\partial K}u,X\rangle \] for the weighted tangential divergence of a vector field $X \in C^1(\R^n,\R^n)$. Notice that, by the divergence theorem on a manifold without boundary, the weighted divergence satisfies
\begin{equation}\label{eq:weighted_div}
    \int_{\partial K} \nabla_{\partial K, \mu} \cdot X \di \mu = 0
\end{equation}
for every $X \in C^1(\R^n,\R^n)$. With this notation, we define on $H^1(\partial K, \mu)$ the operator $\mathcal{L}$ in Theorem \ref{thm:Section_4_main1} by
\begin{equation}\label{def:Boundary_Elliptic_Operator}
\mathcal{L}(\rho)
=
-\nabla_{\partial K,\mu}\cdot({\rm II}^{-1}\nabla_{\partial K}\rho)
-{\rm H}_\mu \rho
+\frac{1}{\mu(K)}\int_{\partial K} \rho\di\mu.
\end{equation}
Thus, the function $\overline{\rho}$ is a weak solution of $\mathcal{L}(\overline{\rho})=1$. Notice that, if $K$ is smooth, then the operator $\mathcal{L}$ is uniformly elliptic since ${\rm II}^{-1}$ is positive definite. Moreover, for every $\rho \in H^1(\partial K, \mu)$ a quick computation shows that
\begin{equation}\label{eq:P_L_identity}
    \int_{\partial K} \rho \mathcal{L}(\rho) \di \mu = \langle \rho, \rho \rangle_{\rm P}.
\end{equation}

For a fixed even measure $\mu$ with smooth and strictly convex potential, \eqref{eq:dim_BM} is equivalent to proving that, for every smooth origin-symmetric convex body $K \in \K^n$, \[ p(\mu,K) \geq \frac{1}{n}.\] We now provide a reformulation of this statement in which, surprisingly, the bilinear form \eqref{eq:I} from \eqref{eq:multiplicative_form} appears. As we shall see in the proof, the interaction between $\mathcal{L}$ and the support function $h_K$ plays a central role. 
\begin{theorem}\label{thm:reformulation}
    Let $K$ and $\mu$ satisfy the assumptions of Theorem~\ref{thm:multiplicative_form}. In particular, $\di \mu(x)=e^{-u(x)}\di x$ for a smooth and strictly convex function $u$. Then, 
    \begin{equation}\label{eq:equivalent_dim_BM}
        p(\mu,K) \geq \frac{1}{n} \Longleftrightarrow \langle \rhobar,\langle \nabla u, x \rangle \rangle_{\rm I}+\int_K \langle \nabla u, x \rangle \di \mu \geq 0,
    \end{equation}
    where $\rhobar$ is the unique solution of $\mathcal{L}(\rho)=1$ on $\partial K$.
\end{theorem}
\begin{proof}
We start by observing that 
\begin{equation}\label{eq:integral_support}
    \int_{\partial K} h_{K}(\nu_{K})\,\di \mu
= n\,\mu(K) -\int_{K}\langle \nabla u,x\rangle \,\di \mu.
\end{equation}
Similar statements have appeared in the literature. We provide a proof for the convenience of the reader. Indeed, for $x\in \partial K$ one has $h_{K}(\nu_{K}(x))=\langle x,\nu_{K}(x)\rangle$. The divergence theorem yields
\begin{align*}
&\int_{\partial K}h_{K}(\nu_{K})\,\di\mu
=\int_{\partial K} \langle x, \nu_{K}\rangle e^{-u}\,\di x=
\int_{K} \nabla \cdot \left(x e^{-u}\right)\,\di x \\
=&\int_{K} n\,\di\mu -\int_{K}\langle \nabla u,x\rangle\,\di \mu 
=n\mu(K)-\int_{K}\langle \nabla u(x),x\rangle\,\di \mu,
\end{align*}
proving \eqref{eq:integral_support}.

Next, we claim that
\begin{equation}\label{eq:support_calculation}
\mathcal{L}\big(h_{K}(\nu_{K})\big) = 1+\langle \nabla u,x\rangle -\frac{1}{\mu(K)}\int_{K} \langle \nabla u, x\rangle \,\di \mu.
\end{equation}
First, observe that, by the properties of the support function
\begin{align*}
    \nabla_{\partial K} \cdot \left( {\rm II}^{-1} \nabla_{\partial K}(h_K(\nu_K))\right)= & \nabla_{\partial K}\cdot \left( \nabla_{\sph^{n-1}}h_K(\nu_K)\right)\\
    = \tra\left( \Id -h_K(\nu_K){\rm II}\right)= &(n-1)-h_K(\nu_K)\tra({\rm II}).
\end{align*}
Therefore,
\begin{equation}\label{eq:weighted_part}
\nabla_{\partial K, \mu} \cdot \left({\rm II}^{-1}\nabla_{\partial K}\big(h_{K}(\nu_{K})\big)\right) =(n-1)-h_K(\nu_K)\tra({\rm II})
-\langle \nabla_{\partial K}u, x\rangle .
\end{equation}
Combining \eqref{eq:weighted_part} with the definition of $\mathcal{L}$, the definition of ${\rm H}_\mu$, and \eqref{eq:integral_support}, we obtain
\begin{align*}
\mathcal{L}\big(h_{K}(\nu_{K}(x))\big)
&=
1+\langle \nabla_{\partial K}u(x), x\rangle
+\langle \nabla u(x), \nu_{K}(x)\rangle\,h_{K}(\nu_{K}(x))
-\frac{1}{\mu(K)} \int_{K} \langle \nabla u(x), x\rangle \,\di \mu(x) \\
&=
1+\langle \nabla u(x), x\rangle
-\frac{1}{\mu(K)}\int_{K} \langle \nabla u(x), x\rangle \,\di \mu(x),
\end{align*}
and \eqref{eq:support_calculation} holds. 

Observe now that, analogously to \eqref{eq:weighted_div}, the operator $\mathcal{L}$ is symmetric in the sense that for any sufficiently regular functions $\rho,\sigma: \partial K \to \R$,
\[
\int_{\partial K} \rho\mathcal{L}(\sigma)\di \mu
=
\int_{\partial K} \mathcal{L}(\rho)\,\sigma\,\di \mu.
\]
Using this property together with the fact that $\mathcal{L}(\overline{\rho})=1$ in the weak sense, we obtain
\begin{align*}
\int_{\partial K}h_{K}(\nu_{K})\,\di\mu
-
\int_{\partial K}\overline{\rho}\,\di \mu
=
\int_{\partial K} \big(h_{K}(\nu_{K})-\overline{\rho}\big)\,\mathcal{L}(\overline{\rho})\,\di \mu 
=
\int_{\partial K} \mathcal{L}\big(h_{K}(\nu_{K})-\overline{\rho}\big)\,\overline{\rho}\,\di \mu.
\end{align*}
A direct application of \eqref{eq:support_calculation} yields
\begin{equation} \label{eq:thm_dim_ref}
    \int_{\partial K}h_{K}(\nu_{K})\,\di\mu- \int_{\partial K} \overline{\rho}\,\di \mu = \langle \overline{\rho}, \langle \nabla u , x \rangle \rangle_{\rm I}.
\end{equation}

Finally, by \eqref{eq:new_concavity expression} and \eqref{eq:integral_support}, $p(\mu,K)\geq 1/n$ is equivalent to
\[n\mu(K) \geq \int_{\partial K}\rhobar \di \mu  \Leftrightarrow \int_{\partial K} h_K(\nu_K) \di \mu +\int_{K} \langle \nabla u,x \rangle \di \mu \geq \int_{\partial K} \rhobar \di \mu,\]
which, in turn, by \eqref{eq:thm_dim_ref} is equivalent to \[\langle \rhobar,\langle \nabla u, x \rangle \rangle_{\rm I}+\int_K \langle \nabla u, x \rangle \di \mu \geq 0,\] concluding the proof.
\end{proof}

We conclude this section with a further remark on the role of $\mathcal{L}$ and $h_K$. Consider indeed \eqref{eq:support_calculation} in the case $u\equiv 0$. Notice that the operator $\mathcal{L}$ is still well-defined. Then, immediate computations show that \[ \mathcal{L}(h_K)=1. \] This is equivalent to the fact that in Colesanti's version of \eqref{eq:Poincare} (that is, when $\mu$ is the Lebesgue measure), the equality case is given by homotheties of $K$. Moreover, denoting by $\vol_n$ the standard Lebesgue measure on $\R^n$, it is well known that \[ \int_{\partial K} h_K(\nu_K)\di \haus^{n-1}=n\vol_n(K),\]
which encodes the dimensional concavity of convex sets, given by the Brunn--Minkowski inequality. In this sense, for a sufficiently smooth log-concave measure $\mu$, the corresponding function $\rhobar$ may be viewed as a weighted analogue of the support function of $K$.

\section{Symmetric case and Hessian pinching}

In this final section, we shall prove Theorem \ref{thm:SG_intro}. To do so, we specialize to the case where the convex body $K$ is origin-symmetric and the measure $\mu$ is even. First we report the following lemma, which is a direct consequence of \cite[Lemma 5.1]{Kol_Liv_Gardner-Zvavitch}.
\begin{lemma}\label{lem:radial_bound}
Consider an origin-symmetric smooth convex body $K \subset \R^n$ and an even function $u \in C^2(\R^n)$ such that
\[
k_{1}\,\mathrm{Id} \leq \nabla^{2}u
\qquad\text{and}\qquad
\Delta u=\tra(\nabla^{2}u)\leq k_{2}n.
\]
Then, if $\mu$ is such that $\di\mu(x)=e^{-u(x)}\di x$,
\begin{equation}\label{eq:measure_moment_estimate}
\frac{1}{\mu(K)}\int_{K}
\langle (\nabla^{2}u)^{-1}\nabla u,\nabla u\rangle\,\di\mu
\leq
n\frac{k_{2}}{k_{1}}.
\end{equation}
\end{lemma}
Combining Theorem~\ref{thm:mean_form} with \eqref{eq:thm_dim_ref}, we obtain the following general bound.
\begin{proposition}\label{prop:concavity_general_bound}
Suppose that $K$ and $\mu$ satisfy the assumptions of Theorem~\ref{thm:multiplicative_form}. Then
\[
\frac{1}{p(\mu,K)}
\leq n+\frac{1}{\mu(K)}\int_{K}\langle (\nabla^{2}u)^{-1}\nabla u,\nabla u\rangle\,\di\mu.
\]
\end{proposition}
\begin{proof}
Let $g(x)=\langle \nabla u(x),x\rangle$ and $f(x)=\overline{\rho}(x)-h_{K}(\nu_{K}(x))$. By \eqref{eq:thm_dim_ref} we infer that
\begin{align*}
\mathcal{L}(f) = \mathcal{L}(\overline{\rho}-h_{K}(\nu_K)) = -\langle \nabla u,x\rangle +\int_{ K} \langle \nabla u,x \rangle \di \mu.
\end{align*}
Multiplying both sides by $f$ and integrating on $\partial K$ with respect to $\mu$, by \eqref{eq:P_L_identity} we deduce $\langle f,f\rangle_{{\rm P}} = -\langle g,f\rangle_{\rm I}$. Now, the choice $\rho = f$ and $\varphi = -g$ in Theorem~\ref{thm:mean_form} yields
\begin{equation}\label{eq:5.2.1}
\langle f,f\rangle_{{\rm P}} = -\langle g,f\rangle_{\rm I} \geq -\frac{\langle f,f\rangle_{{\rm P}}+\langle g,g\rangle_{{\rm BL}}}{2} \implies -\langle f,f\rangle_{{\rm P}} = \langle g,f\rangle_{\rm I} \geq -\langle g,g\rangle_{{\rm BL}}.
\end{equation}
Furthermore, notice that through an integration by parts and using \eqref{eq:integral_support}
\begin{align}\label{eq:5.2.2}
    \notag \langle f,h_{K}(\nu_K)\rangle_{{\rm I}} &= \int_{\partial K} \langle \nabla u,x \rangle h_{K}(\nu_{K}) \di \mu-\frac{1}{\mu(K)} \int_{\partial K}h_{K}(\nu_{K})\di \mu\int_{K}\langle \nabla u,x \rangle \di \mu \notag \\
    \notag &= n\int_{K}\langle \nabla u,x \rangle \di \mu -\int_{K}\langle \nabla u,x \rangle^2 \di \mu+\int_{K}\langle (\nabla^2u)x,x\rangle \di \mu+ \int_{K}\langle \nabla u,x \rangle \di \mu\\ \notag &-n\int_{K}\langle \nabla u,x \rangle\di \mu +\frac{1}{\mu(K)}\left(\int_{K} \langle \nabla u, x \rangle\right)^2 \di \mu\notag \\
    &= \int_{K}\langle (\nabla^2u)x,x\rangle \di \mu+\int_{K}\langle \nabla u,x\rangle\di \mu - \int_{K}\langle \nabla u,x \rangle^2\di \mu+\frac{1}{\mu(K)}\left(\int_{K} \langle \nabla u, x \rangle \di \mu\right)^2.
\end{align}
Combining \eqref{eq:5.2.1} and\eqref{eq:5.2.2} with \eqref{eq:thm_dim_ref}, we infer 
\begin{align}
    \notag n\mu(K) - \int_{\partial K }\rhobar \di \mu&= \int_{K}\langle \nabla u, x \rangle \di \mu+ \langle g,f\rangle_{{\rm I}}+\langle g,h_{K}\rangle_{{\rm I}} \geq \int_{K}\langle \nabla u, x \rangle \di \mu- \langle g,g\rangle_{{\rm BL}}+\langle g,h_{K}\rangle_{{\rm I}} \\ \notag
    &= \int_{K}\langle (\nabla^2u)x,x\rangle \di \mu+2\int_{K}\langle \nabla u,x\rangle \di \mu- \int_{K}\langle \nabla u,x \rangle^2\di \mu+\frac{1}{\mu(K)}\left(\int_{K} \langle \nabla u, x \rangle \di \mu\right)^2 \\
    \notag &-\int_{K}\langle (\nabla^2u)x,x\rangle\di \mu-2\int_{K}\langle \nabla u,x\rangle\di \mu+\int_{K}\langle \nabla u,x \rangle^2\di \mu-\frac{1}{\mu(K)}\left(\int_{K} \langle \nabla u, x \rangle \di \mu\right)^2 \\
    &-\int_{K}\langle (\nabla^{2}u)^{-1}\nabla u,\nabla u\rangle \di \mu. \notag
\end{align}
That is, \[ \int_{\partial K }\rhobar \di \mu \leq n \mu(K) + \int_{K}\langle (\nabla^{2}u)^{-1}\nabla u,\nabla u\rangle \di \mu, \]
and the proof is concluded by recalling that $\frac{1}{p(\mu,K)} = \frac{1}{\mu(K)}\int_{\partial K} \rhobar \di \mu$.
\end{proof}

We remark that the above estimate cannot lead to a resolution of the dimensional Brunn--Minkowski conjecture, since even applying the main inequality from Theorem~\ref{thm:multiplicative_form} already yields a bound strictly larger than $n$. Nevertheless, the estimate holds for an arbitrary convex body $K$, which need not contain the origin, and for an arbitrary smooth log-concave measure $\mu=e^{-u}$. In this general setting, we obtain control of the concavity power in terms of the quantity
\[
 \frac{1}{\mu(K)}\int_{K}(\nabla^{2}u)^{-1}\nabla u,\nabla u\rangle\di\mu.
\]
Assuming, in addition, that $K$ is symmetric and that the Hessian of $u$ is pinched, we can finally prove Theorem \ref{thm:SG_intro}. We repeat the statement for the convenience of the reader.
\begin{theorem}
Consider an origin-symmetric smooth convex body $K \subset \R^n$ and an even strictly convex function $u \in C^2(\R^n)$ such that
\[
k_{1}\,\mathrm{Id} \leq \nabla^{2}u \,\,\text{ and } \,\,  \Delta u \leq k_{2}\,n
\]
for some $0 < k_{1} \leq k_{2}$. If $\mu$ is such that $\di\mu(x)=e^{-u(x)}\di x$, setting $r=\frac{k_{2}}{k_{1}}$ and $c=\frac{1}{r+1}$ we have
\begin{equation}\label{eq:radial_estimate}
p(\mu,K) \geq \frac{c}{n}.
\end{equation}
In particular, 
\[ \mu((1-t)K+tL)^{\frac{c}{n}}\geq (1-t)\mu(K)^{\frac{c}{n}}+t\mu(L)^{\frac{c}{n}}\] for every origin-symmetric $K, L \in \K^n$ and $t \in [0,1]$
\end{theorem}
\begin{proof}
The proof is a straightforward application of Proposition~\ref{prop:concavity_general_bound} together with Lemma~\ref{lem:radial_bound}.
\end{proof}

\bigskip

\paragraph{Acknowledgments.} The first-named author was supported by the Austrian Science Fund (FWF): 10.55776/P36344N. The second-named author was supported, in part, by the Austrian Science Fund (FWF): 10.55776/PAT3787224.

\footnotesize
\bibliography{references}
\bibliographystyle{siam}

\parbox[t]{8.5cm}{
Sotiris Armeniakos\\
Institut f\"ur Stochastik und Wirtschaftsmathematik\\
TU Wien\\
Wiedner Hauptstra{\ss}e 8-10/1046\\
1040 Wien, Austria\\
e-mail: sotirios.armeniakos@tuwien.ac.at}

\medskip

\parbox[t]{8.5cm}{
Jacopo Ulivelli\\
Institut f\"ur Diskrete Mathematik und Geometrie\\
TU Wien\\
Wiedner Hauptstra{\ss}e 8-10/1046\\
1040 Wien, Austria\\
e-mail: jacopo.ulivelli@tuwien.ac.at}
\end{document}